\documentclass[12pt]{amsart}
\usepackage{amsmath}
\usepackage{amssymb}
\usepackage{amsfonts}
\usepackage{mathrsfs}
\usepackage{fullpage}
\usepackage{stmaryrd}
\usepackage{url}
\usepackage[small,nohug,heads=vee]{diagrams}

\newcommand{\g}{\mathfrak{g}}

\newcommand{\mfk}{\mathfrak{k}}

\newcommand{\mft}{\mathfrak{t}}
\newcommand{\mfb}{\mathfrak{b}}
\newcommand{\mfn}{\mathfrak{n}}

\newcommand{\mfu}{\mathfrak{u}}
\newcommand{\mfl}{\mathfrak{l}}

\newcommand{\mfs}{\mathfrak{s}}
\newcommand{\bfs}{\mathbf{s}}

\newcommand{\bbC}{\mathbb{C}}

\newcommand{\bbR}{\mathbb{R}}
\newcommand{\cO}{\mathcal{O}}

\newcommand{\Id}{\mathrm{Id}}
\newcommand{\sB}{\mathscr{B}}

\newcommand{\sP}{\mathscr{P}}

\newcommand{\sV}{\mathscr{V}}

\newcommand{\bbP}{\mathbb{P}}

\newcommand{\BGB}{{B \backslash G / B}}
\newcommand{\PGB}{{P \backslash G / B}}
\newcommand{\KGB}{{K \backslash G / B}}
\newcommand{\KGP}{{K \backslash G / P}}
\newcommand{\HGB}{{H \backslash G / B}}
\newcommand{\alphato}{\overset{\alpha}{\mapsto}}
\newcommand{\alphafrom}{\overset{\alpha}{\mapsfrom}}
\newcommand{\Mv}{{}^{v_0}M}
\newcommand{\Mg}{{}^{g}M}

\DeclareMathOperator{\Int}{\mathrm{int}}
\DeclareMathOperator{\Aut}{\mathrm{Aut}}

\DeclareMathOperator{\Ad}{\mathrm{Ad}}

\newtheorem{theorem}{Theorem}[section]
\newtheorem{proposition}[theorem]{Proposition}
\newtheorem{lemma}[theorem]{Lemma}
\newtheorem{corollary}[theorem]{Corollary}

\theoremstyle{definition}
\newtheorem{definition}[theorem]{Definition}
\newtheorem{notation}[theorem]{Notation}

\theoremstyle{remark}

\newtheorem{remark}[theorem]{Remark}

\begin{document}

\title{Simplifying and Unifying Bruhat Order for $\BGB$, $\PGB$, $\KGB$, and 
$\KGP$}

\begin{abstract}
This paper provides a unifying and simplifying approach to Bruhat order in 
which the usual 
Bruhat order, parabolic Bruhat order, and Bruhat order for symmetric 
pairs are shown to have combinatorially analogous and relatively 
simple descriptions.  Such analogies are valuable as they permit the study of 
$\PGB$ and $\KGB$ by reducing to $\BGB$ rather than by introducing additional  
machinery.  A concise definition for reduced expressions and a simple proof of 
the exchange condition for $\PGB$ are provided as applications of this 
philosophy.  A geometric argument for spherical subgroups, which includes 
all of the cases considered, shows 
that Bruhat order has property $Z$ and therefore satisfies the subexpression 
property.  Thus, Bruhat order can be described using only simple relations, and 
it is the simple relations which we simplify combinatorially.
A parametrization of $\KGP$ is a simple consequence of 
understanding the Bruhat order of $\KGB$ restricted to a $P$-orbit.
In $\PGB$, if $x, y \in W_G$ are maximal length $W_L \backslash W_G$ coset 
representatives, then $W_L x \leq W_L y \iff x \leq y$.  Similarly, viewing 
$KuP$ and $KvP$ as unions of $B$ orbits for which $KuB$ and $KvB$ have maximal 
dimension, then $KuP \preceq KvP \iff KuB \preceq KvB$. $\KGP$ is in order 
preserving bijection with the $P$-maximal elements of $\KGB$ in the same way 
that $\PGB$ is in order preserving bijection with the $P$-maximal elements of 
$\BGB$.
\end{abstract}

\author{Wai~Ling~Yee}
\address{Department of Mathematics and Statistics \\ 
University of Windsor \\
Windsor, Ontario \\
CANADA}
\email{wlyee@uwindsor.ca}
\thanks{The author is grateful for the support from a Discovery 
Grant and UFA from NSERC, NSF grants DMS-0554278 and DMS-0968275, and the 
hospitality of the American Institute of Mathematics.}
\dedicatory{This paper is dedicated to the memory of Angela Sodan.}

\maketitle

\section{Introduction}
Bruhat order is an important tool in many branches of representation 
theory, in part because of the importance of studying orbits on the flag 
variety.  Category $\cO$ and the category of Harish-Chandra modules (see 
\cite{K} p. 375) are two 
categories for which representations are related to orbits on the flag variety.
Let $G$ be a complex reductive linear algebraic group with Lie algebra $\g$, 
$\theta$ a Cartan involution of $G$ (specifying a real form), $K = G^\theta$, 
$B$ a $\theta$-stable Borel subgroup, and $P$ a standard parabolic subgroup 
containing $B$.  Using Beilinson-Bernstein's 
geometric construction, irreducible representations in 
Category $\cO$ of trivial infinitesimal character are known to be 
in correspondence with $B$-orbits on the flag variety while irreducible 
Harish-Chandra modules of trivial infinitesimal character are in 
bijection with $K$-equivariant local systems on $K$-orbits on the flag variety.
The module constructions may be modified suitably to produce modules for other 
infinitesimal 
characters.

Multiplicities of irreducible composition factors in standard modules 
for each of these categories can be computed using 
Kazhdan-Lusztig-Vogan polynomials.  Finding 
efficient means of computing such polynomials is a heavily studied problem.  
Since the recursion formulas for computing Kazhdan-Lusztig-Vogan polynomials 
are expressed in terms of the Bruhat order on orbits and on local systems, we 
hope that the simplifications to Bruhat order for $\PGB$ and for $\KGB$ 
contained in this paper may lead to a deeper understanding of the 
Kazdhan-Lusztig-Vogan polynomials in these categories and the relationships 
among them.  (Recall that local systems are parametrized by certain orbits in 
pairs of flag varieties.)

Beyond parametrizing representations of various categories, orbits on 
the flag variety and Bruhat order appear 
in geometry (symmetric spaces, spherical homogeneous spaces) and in number 
theory problems in which one studies the fixed points of an involution.
Bruhat order is ubiquitous in mathematics and is of fundamental importance.

We begin this paper by showing that Bruhat order can be described using only 
simple relations by:
\begin{enumerate}
\item a simple geometric argument for spherical subgroups, which includes all 
of the cases considered, showing that Bruhat order satisfies property $Z$ 
(see \cite{D2} Theorem 1.1 and \cite{RS} Property 5.12(d) or definition 
\ref{PropertyZ})
\item using property $Z$ to show that Bruhat order has the subexpression 
property.
\end{enumerate}

Simple relations for Bruhat order are examined from the 
following perspectives for each of $B \backslash G / B$, 
$P \backslash G / B$, and $K \backslash G/B$:  
\begin{itemize}
\item topological (closure order)
\item cross actions and Cayley transforms
\item roots and the Weyl group
\item roots and pullbacks.
\end{itemize}
Strong analogies are drawn between the different cases 
$\BGB$, $\PGB$, and $\KGB$.  This permits definitions of 
standard objects and proofs of properties to be simplified for $\PGB$ and 
for $\KGB$: it is more efficient to exploit similarities with $\BGB$ 
than it is to introduce new machinery to accommodate the differences.
The simplest combinatorial descriptions of the
simple Bruhat relations are Theorems \ref{BGBBruhat}, \ref{PGBBruhat}, and 
\ref{KGBBruhat}.

This paper is structured as follows.
Section \ref{Notation} contains notation and the setup which will remain 
fixed for the duration of the paper.
In section \ref{Reduction} we show that Bruhat order for spherical subgroups 
can be described using only simple relations.
Sections \ref{BGB} and \ref{PGB} discuss Bruhat order for $B\backslash G / B$ 
and for $P \backslash G / B$, respectively.

We discuss reduced expressions and the exchange property in sections 
\ref{Reduced} and \ref{Exchange} for each of $\BGB$ and $\PGB$.

In section \ref{KGBSection}, Bruhat order for $K \backslash G / B$ is 
simplified and 
shown to be analogous to Bruhat orders for $B \backslash G / B$ and for $P 
\backslash G / B$.  

In section \ref{KGP}, we give a simple combinatorial parametrization of $\KGP$
(Theorem \ref{KGPParam}).
In the same way that $\PGB$ is in bijection with maximal length 
coset representatives, $\KGP$ is in bijection with ``$P$-maximal'' elements 
of $\KGB$.
We show that if $KuB$ and $KvB$ are ``$P$-maximal'', then $KuP \preceq KvP \iff 
KuB \preceq KvB$.  We discuss how the monoidal action descends (or fails to 
descend) from $\KGB$ to $\KGP$.  

The final section contains a discussion of future work.

\subsection{Acknowledgements}  I would like to thank Annegret Paul and 
Siddhartha Sahi from whom I have learned a tremendous amount.  I would also 
like to thank Athony Henderson and John Stembridge for helpful feedback and 
David Vogan for his comments and his suggestion to look for the simplest 
possible geometric explanation that Bruhat order depends only on simple 
relations.

\section{Setup and Notation} \label{Notation}
The following notation will be fixed for the duration of this paper.
\begin{itemize}
\item $G$: complex reductive linear algebraic group
\item $\g$:  the Lie algebra of $G$.  Analogous notation will be used for 
Lie algebras of other groups.
\item $\theta \in \Aut(G)$: an algebraic automorphism of order two
\item $K = G^\theta$
\item $B=TU$:  the Levi decomposition of a Borel subgroup.  We may assume that 
we selected both $B$ and its Levi decomposition $B=TU$ to be $\theta$-stable 
(exists by Steinberg, see \cite{S} 2.3). Therefore $T$ is maximally compact, 
i.e.  $\dim \mft \cap \g^\theta$ is maximal so all of the roots with respect to 
$\mft$ have non-trivial restriction to $\mft^\theta$ (see Proposition 6.70 of 
\cite{K}).
\item  $\Delta( \g, \mft )$:  the roots of $\g$ with respect to $\mft$
\item $W = W_G$:  the Weyl group $N_G(T) / T$ of $G$ 
\item $\Pi$:  the set of simple roots corresponding to $B$
\item $S$:  the set of simple reflections corresponding to $\Pi$
\item $P_J$:  the standard parabolic subgroup corresponding to $J \subset \Pi$
\item $P = LN$:  the $T$-stable Levi decomposition of $P$
\item $I \subset \Pi$: the subset corresponding to $P$
\item $W_L$:  the Weyl Group $N_L(T) / T$ of $L$
\item $x_\alpha: \bbR \to G$:  for the simple root $\alpha$, the
one-parameter subgroup of $G$ associated to $\alpha$.  This satisfies 
$t x_\alpha( \tau ) t^{-1} = x_\alpha( t \tau t^{-1} )$ for all $t, \tau \in T$.
Then $x_{-\alpha} : \bbR \to G$ is chosen to be the unique one-parameter 
subgroup such that $x_\alpha(1)x_{-\alpha}(-1) x_\alpha(1) \in N(T)$.
\item $\phi_\alpha:  SL_2 \to G$: for the simple root $\alpha$, the group 
homomorphism satisfying
        $$x_\alpha( m ) = \phi_\alpha
        \left(\begin{array}{cc}
        1 & m \\ 0 & 1
        \end{array}\right), \qquad
        x_{-\alpha}( m ) = \phi_\alpha
        \left(\begin{array}{cc}
        1 & 0 \\ m & 1
        \end{array}\right), \qquad
        \alpha^\vee(t) = \phi_\alpha
        \left(\begin{array}{cc}
        t & 0 \\ 0 & t^{-1}
        \end{array}\right). $$
\item $\dot{s}_\alpha = x_\alpha(-1) x_{-\alpha}(1) x_\alpha(-1) =
        \phi_\alpha \left(\begin{array}{cc}
        0 & -1 \\ 1 & 0
        \end{array}\right) \in N(T)$
for the simple root $\alpha$ ($\dot{s}_\alpha = n_\alpha$ in the notation of 
\cite{RS})
\item $\dot{w}$: Given $w = s_{i_1} s_{i_2} \cdots s_{i_k} \in W$ a reduced 
expression, define $\dot{w} = \dot{s}_{i_1} \dot{s}_{i_2} \cdots \dot{s}_{i_k} 
\in N(T)$.  It is known that $\dot{w}$ is independent of the reduced expression 
chosen.
\item $H_g := gHg^{-1}$ for every $g \in G$, $H \leq G$
\item $\sB :=$ the variety of Borel subgroups of $G$.  This is in bijection 
with $G / B$ where $B_g \leftrightarrow gB$.
\end{itemize}

\section{Spherical Subgroups:  Reducing Bruhat Order to its Simple 
Relations} \label{Reduction}
In this section, we prove that Bruhat order on $\BGB$, $\PGB$, and $\KGB$ may 
be described using only simple relations; more generally, if $H$ is a spherical 
subgroup of $G$, then Bruhat order on $\HGB$ can be described 
using only simple relations.  For readers far more accustomed to Bruhat order 
on the Weyl group than Bruhat order on $\KGB$, it might be easier to read this 
section last.

Often, one first encounters Bruhat order as a partial order on the Weyl group 
and on quotients of the Weyl group.  That definition of Bruhat order does not 
immediately lend itself to generalization.  Since our goal is to unify Bruhat 
order for $\BGB$, $\PGB$, $\KGB$, and for $\KGP$, before we obtain 
combinatorially analogous definitions of Bruhat order, we focus on orbits on 
the flag variety $G/B$ and use the equivalent topological definition of 
Bruhat order which is common to all of our settings.  Our subgroups $B$, $P$, 
and $K$ share the common property that they are all spherical subgroups of $G$.

\begin{definition}
A closed subgroup $H$ of $G$ is spherical if it has an open orbit 
in $G/B$.
\end{definition}
It is well-known that if $H$ is spherical, then $\HGB$ is 
finite.  This will also follow from the subexpression property.

\begin{definition}
Let $H$ be a spherical subgroup of $G$.
Then the closure order on $\HGB$ is defined by:
$$\cO_1 \preceq^H \cO_2 \quad\text{ if }\quad \cO_1 \subset \overline{\cO}_2.$$
Bruhat order  on $\HGB$ is defined to be closure order.
\end{definition}

It is well-known that under the bijections $\BGB \leftrightarrow W_G$ and 
$\PGB \leftrightarrow W_L \backslash W_G$, the topological definition of 
Bruhat order agrees with the usual definition.

We will see that this 
topological definition can easily be used to find a common proof for all of 
our cases that Bruhat order satisfies property $Z$ and hence satisfies the 
subexpression property.  Thus, Bruhat order for $\BGB$, $\PGB$, and for $\KGB$ 
may be described using only what we will call simple relations.  

\begin{notation}
Given $\alpha$ a simple root, we define $P_\alpha = P_{\lbrace \alpha
\rbrace}$  (the standard parabolic of type $\alpha$ containing $B$).
Then we have the canonical projection:
$$ \pi_\alpha : G / B \to G / P_\alpha.$$
which may be viewed as a projection from  $\sB$ to $\sP_\alpha$, the variety
of parabolics of type $\alpha$.
\end{notation}

The set of Borel subgroups contained in $P_\alpha$ is in bijection with 
$\bbP^1$.  Therefore we see that:

\begin{lemma} (\cite{V6})
$\pi_\alpha$ is a $\bbP^1$-fibration:  $\pi_\alpha^{-1}\pi_\alpha(x) \cong
\bbP^1$ for all $x \in \sB$.
$$
\begin{array}{ccc}
\bbP^1 & \to & G / B \\ 
& & \,\,\,\quad \downarrow  \pi_\alpha \\
& & G / P_\alpha.
\end{array}
$$
\end{lemma}

Because of the $\bbP^1$-fibration,
\begin{corollary}
Let $H$ be a subgroup of $G$ and let $\cO$ be an $H$-orbit on 
$G/B$.  Then
$$\dim \pi_\alpha^{-1} \pi_\alpha ( \cO ) = \dim \cO \text{ or } \dim \cO + 1.$$
\end{corollary}

\begin{lemma}{\cite{Kno}}
For spherical subgroups $H$ of $G$ and $\cO$ an $H$-orbit on $G/B$,
there is a unique dense orbit $\cO' \subset \pi_\alpha^{-1} \pi_\alpha( \cO )$.
$\pi_\alpha^{-1}\pi_\alpha( \cO )$ is a union of 1, 2, or 3 orbits.
\end{lemma}
\begin{proof}
This is a consequence of the geometry of each $\bbP^1$-fibre.  Here, we 
outline Knop's work for its relevance to finiteness of $\HGB$.

Given a $B$-variety, the complexity of the variety is defined to be the 
minimal codimension of a $B$-orbit 
in the variety (\cite{Kno}, p. 287).  Thus, since $H$ is spherical, the 
complexity of $G/H$ is $0$.  By Theorem 2.2 of \cite{Kno}, any $B$-stable 
subvariety of $G/H$ also has complexity $0$.  Therefore $\pi_\alpha^{-1} 
\pi_\alpha( \cO )$ has a codimension $0$ $H$-orbit, $\cO'$.

Let $x \in G$ be such that $B_x \in \cO'$.  Using the notation of this paper, 
there is a correspondence (\cite{Kno} p. 290) between $H \cap 
(P_\alpha)_x$-orbits on 
$(P_\alpha)_x / B_x \cong \bbP^1$ and the $H$-orbits in $\pi_\alpha^{-1}
\pi_\alpha( \cO' ) = \pi_\alpha^{-1}\pi_\alpha( \cO ) = H \backslash x P_\alpha
/B$ via $H \cap (P_\alpha)_x x p x^{-1} x B \mapsto H xp B$ for $p \in 
P_\alpha$.  This 
map is well-defined and surjective.  To show injectivity, suppose that
$H xp_1 B = H xp_2 B$ for some $p_i \in P_\alpha$.  Then we may assume that 
$x p_1 = h x p_2 b$ for some $h \in H$ and $b \in B$.  Rearranging, we see that
$h \in H \cap (P_\alpha)_x$, whence $H \cap (P_\alpha)_x xp_1 B
= H \cap (P_\alpha)_x x p_2 B$.

Focusing on the action of $H \cap (P_\alpha)_x$ on $(P_\alpha)_x / B_x \cong 
\bbP^1$, we may instead study the action of $\bar{H}$ which denotes the 
quotient of $H \cap (P_\alpha)_x$ by the kernel of the action on $\bbP^1$.
Then $\bar{H}$ may be viewed as a subgroup of $\Aut( \bbP^1 ) \cong 
PSL_2( \bbC )$.  According to \cite{Kno} Lemma 3.1, since the complexity of 
$G/H$ is $0$, therefore $\bar{H}$ has positive dimension.  It is easy to 
classify the orbits of positive dimension subgroups of $PSL_2( \bbC )$ on 
$\bbP^1$, and thus $\pi_\alpha^{-1}\pi_\alpha( \cO )$ is a union of 1, 2, or 
3 orbits with $\cO'$ the unique dense orbit (see \cite{Kno} p. 291).
\end{proof}

\begin{definition}
If in the previous corollary $\dim \cO' = \dim \cO + 1$, then write 
$$\cO \alphato \cO' \text{ or } \cO \preceq^H_\alpha \cO'.$$
This is a simple relation for Bruhat order.
\end{definition}

\begin{remark} \label{RemarkAlphaTo}
\begin{enumerate}
\item For $\BGB$, we will see that for some $w \in W_G$, $\cO = BwB$ and 
$\cO' = Bws_\alpha B$ with $\ell( ws_\alpha ) = \ell(w) + 1$.
\item Since $\cO$ is an arbitrary orbit in $\pi_\alpha^{-1}\pi_\alpha(\cO')$ 
different from $\cO'$, therefore $\cO \alphato \cO'$ for any orbit $\cO$ 
different from $\cO'$ in $\pi_\alpha^{-1}\pi_\alpha( \cO' )$.
\end{enumerate}
\end{remark}

\begin{lemma}\label{TopologyLemma}
Let $P_1 \leq P_2 \leq G$ be parabolic subgroups.  Since parabolic subgroups 
are closed, each $G / P_i$ is equipped with the 
quotient topology.  Then, 
letting $p_i : G \to G / P_i$ and $\pi:  G/P_1 \to G / P_2$ be the natural 
projections, we have the commutative diagram:
$$
\begin{diagram}
G & \rTo^{p_2} & G / P_2 \\
\dTo^{p_1} & \ruTo_{\pi} & \\
G / P_1.
\end{diagram}
$$
Then given $X_i \subset G / P_i$:
\begin{enumerate}
\item $\pi( \overline{X_1} ) = \overline{\pi(X_1)}$
\item $\pi^{-1}( \overline{X_2} ) = \overline{\pi^{-1}(X_2)}$
\item $\pi^{-1}\pi( \overline{X_1} ) = \overline{ \pi^{-1}\pi(X_1)}$.
\end{enumerate}
\end{lemma}
\begin{proof}
Each $G/P_i$ is a projective variety, and hence complete.  Then each 
$\pi( \overline{X_1} )$ must be closed (see \cite{S2}, 6.1.2), whence 
$\pi( \overline{X_1} ) \supset \overline{\pi(X_1)}$.
(Note that this containment could also follow from Lemma 2, p. 68, of 
\cite{St}, just as Lemma 4.5 of \cite{RS} did.)

For the opposite containment, since the topology on $G/P_i$ is the quotient 
topology, for $X_1 \subset G / P_1$:
\begin{eqnarray*}
\overline{\pi( X_1 )} &=& \overline{p_2( p_1^{-1}(X_1) )}  \\
&\supset& \bigcap_{V \supset p_1^{-1}(X_1) \text{ closed}} p_2(V) 
\qquad \text{by definition of quotient topology} \\
&=& \bigcap_{V \supset p_1^{-1}(X_1) \text{ closed}} \pi \circ p_1( V ) \\
&\supset& \bigcap_{X \supset X_1 \text{ closed}} \pi( X ) \\
&=& \pi( \overline{X_1} ).
\end{eqnarray*}

Given $X_2 \subset G / H_2$, repeating the above formulas with $p_1$ in place 
of $p_2$, $p_2$ in place of $p_1$, and $\pi^{-1}$ in place of $\pi$ yields 
$\overline{\pi^{-1}( X_2 )} \supset \pi^{-1}( \overline{X_2})$.  We have the 
opposite containment since $\pi$ is continuous.

The final statement follows from the first two statements.
\end{proof}

\begin{definition} \label{PropertyZ} (\cite{D2}, \cite{RS})
Let $H$ be a spherical subgroup of $G$ and
consider Bruhat order on 
$\HGB$.  Let $\alpha$ be a simple root and let $u_1, u_2, v_1, 
v_2 \in H \backslash G / B$ with $u_1 \alphato u_2$ and $v_1 \alphato v_2$.  We 
say that Bruhat order satisfies property $Z(\alpha, u_1, v_1)$ if the 
following are equivalent:
\begin{enumerate}
\item \label{Z1}
$u_1 \preceq^H v_1$ or there exists $x$ such that $x \alphato u_2$ and $x \preceq^H v_1$
\item \label{Z2}
$u_2 \preceq^H v_2$
\item \label{Z3}
$u_1 \preceq^H v_2$.
\end{enumerate}
Bruhat order is said to satisfy property $Z$ if it satisfies every property 
$Z(\alpha,u_1,v_1)$.
\end{definition}

\begin{theorem}
If $H$ is a spherical subgroup of $G$,
then Bruhat order on $H \backslash G / B$ satisfies 
property $Z$.
\end{theorem}
\begin{proof}
We use the notation of the previous definition, replacing $u_1$ with 
$\cO_{u_1}$, for example, to remind us that we are studying $H$-orbits.

\noindent
(\ref{Z1}) $\Rightarrow$ (\ref{Z2}):  Let $x \preceq^H v_1$ be such that 
$x \alphato u_2$, possibly setting $x = u_1$.  Applying Lemma 
\ref{TopologyLemma} with $H_1 = B$, $H_2 = P_\alpha$,
$$\begin{array}{ccc}
\pi_\alpha^{-1}\pi_\alpha( \cO_{x}) & \subset \pi_\alpha^{-1} \pi_\alpha( 
\overline{\cO_{v_1}}) = & \overline{\pi_\alpha^{-1}\pi_\alpha( \cO_{v_1} )} \\
\cup &  & \parallel \\
\cO_{u_2} & & \overline{\cO_{v_2}}
\end{array}
$$
since $\cO_{v_2}$ is dense in $\pi_\alpha^{-1}\pi_\alpha( \cO_{v_1} )$.

\noindent
(\ref{Z2}) $\Rightarrow$ (\ref{Z3}):
It suffices to note that $\cO_{u_1} \subset \overline{\pi_\alpha^{-1}
\pi_\alpha( \cO_{u_1})} = \overline{\cO_{u_2}}$ since $\cO_{u_2}$ is 
dense in $\pi_\alpha^{-1}\pi_\alpha( \cO_{u_1} )$.

\noindent
(\ref{Z3}) $\Rightarrow$ (\ref{Z1}):
If $\cO_{u_1} \subset \overline{\cO_{v_2}} = \overline{\pi_\alpha^{-1}
\pi_\alpha( \cO_{v_1})} = \pi_\alpha^{-1}\pi_\alpha( \overline{\cO_{v_1}})$, 
then  $\cO_{u_1} \subset \pi_\alpha^{-1} \pi_\alpha( \cup_{u \preceq^H v_1} 
\cO_u )$.
Therefore 
$\cO_{u_1} \subset \pi_\alpha^{-1} \pi_\alpha( \cO_x )$ for some 
$x \preceq^H v_1$.  If $x = u_1$ or $x=u_2$, then $u_1 \preceq^H v_1$.  
Otherwise, since $\cO_{u_2}$ is the unique dense orbit in $\pi_\alpha^{-1}
\pi_\alpha( \cO_x )$, $x$ satisfies $x \alphato u_2$ in addition to 
$x \preceq^H v_1$, whence (\ref{Z1}) is satisfied.
\end{proof}

\begin{definition}\label{ReducedDecomposition}(\cite{RS}, 5.7)
Let $H$ be a spherical subgroup of $G$.  A pair of sequences $((v_0, v_1, 
\ldots, v_k), (\alpha_1, \ldots, \alpha_k))$, where the $v_i$ are 
$H$-orbits and the $\alpha_i$ are simple roots, is a reduced 
decomposition of $v \in \HGB$ if
$v_0$ is a closed orbit, $v_k = v$, and $v_{i-1} \overset{\alpha_i}{\mapsto} 
v_{i}$ for
$i = 1, \ldots, k$.  The length of the reduced decomposition is $k$.  Note 
that $\dim \cO_v = \dim \cO_{v_0} + k$.
\end{definition}

\begin{definition}
Given $((v_0, \ldots, v_k), (\alpha_1, \ldots, \alpha_k))$ a reduced 
decomposition of $v \in \HGB$, a subexpression of that reduced decomposition 
is a sequence $(u_0, \ldots, u_k)$ such that $u_0 = v_0$ and for $i = 1, 
\ldots, k$:
\begin{enumerate}
\item  $u_i = u_{i-1}$, or
\item $\exists \, u$ such that both $u_{i-1} \overset{\alpha_i}{\mapsto} u$
and $u_{i} \overset{\alpha_i}{\mapsto} u$, or
\item $u_{i-1} \overset{\alpha_i}{\mapsto} u_i$.
\end{enumerate}
\end{definition}
\begin{remark}
We see by the deletion condition that this generalizes the usual definition 
for subexpression for elements of the Weyl group.
\end{remark}

\begin{theorem}
Let $H$ be a spherical subgroup of $G$.  Then Bruhat order 
for $\HGB$ has the subexpression 
property.  That is, $u \preceq^H v$ if and only if there is a reduced expression 
$((v_0, \ldots, v_k), (\alpha_1, \ldots, \alpha_k))$ of $v$ and a subexpression
$(u_0, \ldots, u_k)$ of the reduced expression such that $u_k = u$.
\end{theorem}
\begin{proof}
Deodhar's proof that property $Z$ implies the subexpression property from 
\cite{D2} works in this general setting.

\noindent
$\Leftarrow:$  We prove by induction on $j$ that for each $0 \leq j \leq k$, 
$u_j \preceq^H v_j$.  The statement holds for $j=0$.
Assume that the statement holds for $j=i-1$.  In case (1) where $u_i = u_{i-1}$,
the statement clearly holds for $j=i$.  In case (3), property $Z$ implies that 
since $u_{i-1} \overset{\alpha_i}{\mapsto} u_i$ and
$v_{i-1} \overset{\alpha_i}{\mapsto} v_i$, then  $u_{i-1} \preceq^H v_{i-1} 
\Rightarrow u_i \preceq^H v_i$.
In case (2), since $u_{i-1} \overset{\alpha_i}{\mapsto} u$ and
$v_{i-1} \overset{\alpha_i}{\mapsto} v_i$, thus  $u_{i-1} \preceq^H v_{i-1} 
\Rightarrow u \preceq^H v_i$.  Then by definition of $u_i$, $u_i \preceq^H u \preceq^H v_i$.

\noindent
$\Rightarrow$:  We prove the converse by induction on the length $k$ of the 
reduced expression.  If $k=0$, $v=v_0$ 
and the only subexpressions are those which terminate in $u_0 = v_0$.
Assume that the converse holds when the reduced expressions have lengths 
less than $k$.  (Note that due to the relationship between reduced expression 
length and dimensions, $u \preceq^H v$ implies that the length of a reduced 
expression for $u$ is less than or equal to the length for $v$.)  Now, $v_{k-1} 
\overset{\alpha_k}{\mapsto} v$.  If $u \preceq^H v_{k-1}$, then by induction, there 
is a subexpression for $u$ of a reduced expression for $v_{k-1}$, whence there 
is a subexpression for $u$ of a reduced expression for $v$.  Otherwise, 
there is $u'$ such that $u' \overset{\alpha_k}{\mapsto} u$ and, by 
property $Z$, $u' \preceq^H v_{k-1}$.  By induction, there is a subexpression for 
$u'$ of a reduced expression for $v_{k-1}$. 
By appending $u$ and $v=v_k$, we see that there is a subexpression for $u$ of 
a reduced expression for $v$.
\end{proof}

\begin{corollary}  If $H$ is a spherical subgroup of $G$, then $\HGB$ is finite.
\end{corollary}
\begin{proof}
This was pointed out by Vogan.
Since the dimension of an $H$-orbit is finite, therefore there are finitely 
many reduced expressions for the unique open orbit in $G/B$.  There are 
finitely many subexpressions of each reduced expression, and thus by the 
subexpression property, $\HGB$ is finite.
\end{proof}

\begin{corollary}
Bruhat order for $\BGB$, $\PGB$, and $\KGB$ satisfy the subexpression 
property, whence Bruhat order can be defined using only simple relations.
\end{corollary}

It is the focus of the subsequent chapters to find elementary means of 
describing these simple relations and showing how their descriptions are 
analogous in all of our settings.

\section{Bruhat Order for $B \backslash G / B$} \label{BGB}
\subsection{Equivalence of Closure Order and Bruhat Order on $W$}
As mentioned, it is well-known that $\BGB$ is in bijection with the Weyl 
group:
\begin{proposition}\label{BGBBruhatDecomposition}
Bruhat Decomposition:
$$G = \amalg_{w \in W} B \dot{w} B = \amalg_{w \in W} BwB.$$
Therefore $B \backslash G / B \leftrightarrow W$.
\end{proposition}

We review the definition of Bruhat order on the Weyl group $W_G$ arising 
from viewing it as a reflection group.

\begin{definition}
For $u, v \in W$, $t$ a (not necessarily simple) reflection, we write 
$u \xrightarrow{t} v$ if $v = ut$ and $\ell(u) < \ell(v)$.  Bruhat order on 
$W$ is defined by $u \leq v$ if there exists a sequence of elements $w_0, w_1, 
\ldots, w_k \in W$ such that $u=w_0 \rightarrow w_1 \rightarrow \cdots 
\rightarrow w_k = v$.
\end{definition}

\begin{theorem}
Let $\cO_w = B w B$ for $w \in W$ and let $\alpha$ be a simple root.  
Then 
$$\cO_w \alphato \cO \text{ if and only if } \cO = \cO_{ws_\alpha}
\text{ and } w \xrightarrow{s_\alpha} w s_\alpha. $$
\end{theorem}
\begin{proof}
This is well-known, but we provide a discussion to illuminate the source of the 
similarities between the various double cosets for which we consider Bruhat 
order.
First, $\pi_\alpha^{-1} \pi_\alpha( B w B ) = B w P_\alpha / B
= B w B \cup B ws_\alpha B$ by the parabolic Bruhat 
decomposition (Proposition \ref{PGBBruhatDecomposition}) and by Lemma 8.3.7 
of \cite{S2}.

Recall that $\ell(w) = \# \{ \alpha \in \Delta^+ : w\alpha < 0 \}$.  
Under the correspondence $G/B \leftrightarrow \sB$ where $gB 
\mapsto gBg^{-1} = B_g$, we may view $B_{\dot{w}}$ as a point in the orbit 
$\cO_w$ and $B \cap B_{\dot{w}}$ as the stabilizer of that point.  If 
$\ell(ws_\alpha) = \ell(w) + 1$, then $\dim B \cap B_{\dot{w}} = \dim B \cap 
B_{\dot{w}\dot{s}_\alpha} + 1$.
Since the dimension of a $B$-orbit is the dimension of $B$ minus the dimension 
of the stabilizer of a point in the orbit, $\dim BwB + 1 = \dim B
ws_\alpha B$.
\end{proof}

\begin{remark}
The relationship between the dimension of an orbit and the dimension of the 
stabilizer of a point in the orbit is the source of the similarities between 
descriptions of simple Bruhat relations for our various settings in terms of 
Weyl group elements and positive roots.
\end{remark}

Observing now that the notion of reduced decomposition arising from viewing 
$W$ as a reflection group corresponds with definition 
\ref{ReducedDecomposition}, we conclude:
\begin{corollary}
Under the correspondence arising from the Bruhat decomposition (Proposition 
\ref{BGBBruhatDecomposition}), Bruhat order for $\BGB$ agrees with Bruhat 
order for $W$.
\end{corollary}

We now proceed to reformulate simple relations for Bruhat order.  We find 
that our reformulations apply not only to simple relations.

\subsection{Weyl Group and Roots}
Another means of describing (not necessarily simple) Bruhat relations is:
\begin{theorem} \label{BGBBruhat} 
Let $\alpha$ be a positive root and $w \in W$.  Then:
\begin{eqnarray*}
w & \xrightarrow{s_\alpha} & ws_\alpha \text{ if } w\alpha \in \Delta^+ = 
\Delta( \mfu, \mft ) \\
w & \xleftarrow{s_\alpha}& ws_\alpha \text{ if } w\alpha \in \Delta^- = 
\Delta( \mfu^-, \mft ).
\end{eqnarray*}
\end{theorem}
\begin{proof}
This is known if $\alpha$ is a simple root (\cite{H}, Lemma 1.6).  In general, 
suppose $\ell(w) < \ell(ws_\alpha)$.  Let $w s_\alpha = s_1 s_2 \cdots s_r$ 
be a reduced expression for $ws_\alpha$.  By the strong exchange condition,
$w = s_1 s_2 \cdots \hat{s}_i \cdots s_r$.  Then $s_\alpha = w^{-1} s_1 s_2 
\cdots s_r = s_r s_{r-1} \cdots s_{i+1} s_i s_{i+1} \cdots s_{r-1} s_r$.
We conclude that $\alpha = s_r s_{r-1} \cdots s_{i+1} \alpha_i$.  (Note that 
since $s_1 s_2 \cdots s_r$ is a reduced expression, 
$\alpha = s_r s_{r-1} \cdots s_{i+1} \alpha_i$ is a positive root, 
\cite{H} p. 14.)  Then 
$w\alpha = s_1 \cdots \hat{s}_i \cdots s_r s_r \cdots s_{i+1} \alpha_i
= s_1 \cdots s_{i-1} \alpha_i > 0$ since $s_1 s_2 \cdots s_r$ is a reduced 
expression.

Similarly, if $\ell(w) > \ell(w s_\alpha)$, then $w\alpha < 0$.
\end{proof}

\subsection{Roots and Pullbacks}
Let $\Delta = \Delta( \g, \mft )$ and $\Delta^+ = \Delta( \mfb, \mft)$.
Recall the notation of section \ref{Notation}: $B = TU  \stackrel{\Int(g)}
{\longmapsto} B_g = T_g U_g$.
Considering Lie algebras, $\mfb = \mft \oplus \mfu \stackrel{\Ad(g)}
{\longmapsto} \mfb_g = \mft_g \oplus \mfu_g$.
We have the map between Cartan subalgebras $\Ad(g^{-1}) : \mft_g \to \mft$,
while pullback allows us to map between duals of Cartan subalgebras:
$\Ad(g^{-1})^* : \mft^* \to \mft_g^*$:
        $$(\Ad(g^{-1})^* \alpha)(t) = \alpha( \Ad(g^{-1}) t ) \quad \text{for
        all }t \in \mft_g.$$
Letting $\alpha_g = \Ad(g^{-1})^*\alpha$, it is easy to see that
\begin{eqnarray*}
        \Ad(g) : \mfu = \bigoplus_{\alpha \in \Delta^+} \g_\alpha 
        &\to& \mfu_g = \bigoplus_{\alpha \in \Delta^+} \g_{\alpha_g} \\
\text{with} \qquad \Ad(g) \g_\alpha &=& \g_{\alpha_g}.
\end{eqnarray*}

It is straightforward to prove (use $\Int$ rather than $\Ad$):
\begin{lemma}
For $w \in W$ and $n \in N(T)$ any representative of $w$, 
$ w \alpha = \alpha_n.$  In particular, $w \alpha = \alpha_{\dot{w}}$.
\end{lemma}

Using pullbacks, Bruhat order may be reformulated as follows:
\begin{proposition} 
Let $\alpha$ be a positive root and $w \in W$.  Then:
\begin{eqnarray*}
w & \xrightarrow{s_\alpha}& ws_\alpha \text{ if } \alpha_{\dot{w}} 
\in \Delta( \mfu, \mft ) \\
w & \xleftarrow{s_\alpha}& ws_\alpha \text{ if } \alpha_{\dot{w}} \in 
\Delta( \mfu^-, \mft ).
\end{eqnarray*}
This may also be written
\begin{eqnarray*}
B w B & \alphato & Bws_\alpha B
\text{ if } \alpha_{\dot{w}} \in \Delta( \mfu, \mft ) \\
B w B & \alphafrom & Bws_\alpha B \text{ if } 
\alpha_{\dot{w}} \in \Delta( \mfu^-, \mft ).
\end{eqnarray*}
\end{proposition}

\begin{remark}
Pullbacks turn out to be particularly useful in studying Bruhat order in
more general situations: for example, if one of the subgroups with respect to 
which you take double cosets is twisted by conjugation or if that subgroup is 
a more general spherical subgroup.  See \cite{PSY} for details.
\end{remark}

\subsection{Cross Actions}
\begin{definition}
The cross action of $W$ on $B \backslash G / B$ is the action generated by
$$s_\alpha \times B \dot{w} B := B \dot{w} \dot{s}_\alpha^{-1}B$$
where $\alpha$ is a simple root.
\end{definition}
Under the correspondence between $\BGB$ and $W$, cross action corresponds to 
the natural left action of $W$ on itself by right multiplication by the 
inverse.

It follows immediately:
\begin{theorem}
Let $\alpha$ be a positive root and $w \in W$.  Then:
\begin{eqnarray*}
B {w} B & \alphato & s_\alpha \times B {w} B = 
B {w}{s}_\alpha^{-1} B = B{w}{s}_\alpha B 
\text{ if } w\alpha \in \Delta( \mfu, \mft ) \\
B {w} B & \alphafrom & s_\alpha \times B {w} B = 
B {w}{s}_\alpha^{-1} B = B {w}{s}_\alpha B 
\text{ if } w\alpha \in \Delta( \mfu^-, \mft ). 
\end{eqnarray*}
\end{theorem}

\section{Bruhat Order on $P \backslash G / B$} \label{PGB}
\subsection{Equivalence of Closure Order with Bruhat Order for $W_L \backslash 
W_G$}
It is well-known that:
\begin{proposition} \label{PGBBruhatDecomposition}
(\cite{DM} Proposition 1.6, \cite{H3} Lemma 8.3.7)
Bruhat Decomposition:
$$P = \amalg_{w \in W_L} B \dot{w} B = \amalg_{w \in W_L} BwB $$
and thus $\PGB \leftrightarrow W_L \backslash W_G$.  Furthermore,
$$P w B = \amalg_{x \in W_L} B xw B.$$
\end{proposition}

The commonly used definition for Bruhat order on $W_L \backslash W_G$ viewed 
as a reflection group quotient is this:
\begin{definition}
Bruhat order on $W_L \backslash W_G$ is the partial order induced from 
Bruhat order on $W_G$.  That is, $W_L u \leq W_L v$ if there 
are coset representatives $u_0$ and $v_0$, respectively, such that $u_0 
\leq v_0$.
\end{definition}

That is, if $u, v \in W_G$ are such that $u \leq v$, then $W_L u \leq W_L v$.
The converse holds for minimal length coset representatives:
\begin{proposition} \label{PGBBGBEquiv}
Let $u, v \in W_G$ be minimal length coset representatives for $W_L u$ and 
for $W_L v$.  Then
$$W_L u \leq W_L v \iff u \leq v.$$
\end{proposition}
\begin{proof}
This is a special case of \cite{D2}, Lemma 3.5.
For a purely topological proof using maximal length coset representatives, 
adapt the proof of Proposition \ref{KGPKGBEquiv}.
\end{proof}

Because closure order for $\BGB$ corresponds to Bruhat order for $W_G$, by 
the Bruhat decomposition applied to $P w B$,
closure order for $\PGB$ corresponds to Bruhat order for $W_L \backslash W_G$:
\begin{theorem}
Closure order on $\PGB$ and Bruhat order on $W_L \backslash W_G$ 
correspond.  For $u, v \in W_G$ minimal length coset representatives of 
$W_L u$ and $W_L v$:
$$ \cO_u \preceq^P \cO_v \iff W_L u \leq W_L v.$$
Let $w \in W_G$ and let $\alpha \in \Delta^+( \g, \mft )$.  Using the 
nomenclature of Casian-Collingwood from \cite{CC} for cases which may be shown 
to correspond,
\begin{eqnarray*}
P w B &  = & 
Pw{s}_\alpha B \text{ if } w\alpha \in \Delta( \mfl, \mft ) 
\qquad \text{i.e. } \alpha \text{ is Levi type}\\
P {w} B & \preceq_\alpha^P & 
P{w}{s}_\alpha B \text{ if } w\alpha \in \Delta( \mfn, \mft ) 
\qquad \text{i.e. } \alpha \text{ is complex upward}\\
P {w} B & \succeq_\alpha^P & 
P {w}{s}_\alpha B \text{ if } w\alpha \in \Delta( \mfn^-, \mft )
\qquad \text{i.e. } \alpha \text{ is complex downward.}\\
\end{eqnarray*}
\end{theorem}
\begin{proof}
First, $u \leq v$ implies that $B \dot{u} B \preceq^B B \dot{v}B$ which 
implies that $P \dot{u} B \preceq^P P \dot{v}B$.
Conversely,
\begin{eqnarray*}
P \dot{u} B &\subset & \overline{P \dot{v}B} \\
 \iff \cup_{w \in W_L} B \dot{w} \dot{u} B &\subset & \overline{ \cup_{w \in 
W_L} B \dot{w} \dot{v} B} = 
\cup_{w \in W_L} \overline{B \dot{w}\dot{v} B} \\
 \iff \text{for every } w \in W_L, \, B \dot{w} \dot{u} B &\subset & \overline{
B \dot{x} \dot{v} B} \text{ for some } x \in W_L \\
\Rightarrow u &\leq& xv \text{ for some } x \in W_L \\
\Rightarrow u &\leq& v
\text{ by Lemma 3.5 of \cite{D2} (Proposition \ref{PGBBGBEquiv}).}
\end{eqnarray*}

To prove the first of the remaining three statements, note that 
$P\dot{w}\dot{s}_\alpha B = P \dot{w}\dot{s}_\alpha \dot{w}^{-1} \dot{w} B
= P \dot{s}_{w\alpha} \dot{w} B$.  The rest of the theorem now follows.
\end{proof}

\begin{remark}
This proof does not generalize to $\KGP$ since it relies upon the Bruhat 
decomposition.  We could also 
have proved the theorem using the following more general heuristic.
Bruhat order and closure order on $\PGB$ are the same since:
\begin{enumerate}
\item Bruhat order is induced by Bruhat 
order on $\BGB$ and closure order and Bruhat order for $\BGB$ are the same;
\item the topology on $G/B$ is the quotient topology.
\end{enumerate}
The proof is concise, but rather than record it, we refer the reader to the 
proof of Theorem \ref{KGPClosure} where the order induced from Bruhat order on 
$\KGB$ and closure order of $\KGP$ are shown to be the same.  The proofs are 
similar.
\end{remark}

Recall that Bruhat order for $\PGB$  satisfies property $Z$.  We may restate 
property $Z$ for $W_L \backslash W_G$:
\begin{definition}
Let $u, v \in W_G$ be minimal length coset representatives for $W_L u$ and for 
$W_L v$.  Let $s$ be a simple reflection.  Then {\em property $Z(s, W_L u, 
W_L v )$} is satisfied if whenever 
$\ell( us ) \leq \ell( u )$ and
$\ell( vs ) \leq \ell( v )$, then
$$W_L u \leq W_L v \iff W_L us \leq W_L v \iff W_L us \leq W_L vs.$$
\end{definition}

Furthermore, an analogue of Deodhar's description II for Bruhat order on 
$W_G$ from \cite{D2} holds:
\begin{proposition}
Bruhat order is the unique partial order on $W_L \backslash W_G$ 
such that
\begin{enumerate}
\item for all $W_L w \in W_L \backslash W_G$, $W_L w \leq W_L 1 \iff 
W_L w = W_L 1$;
\item  $\leq$ has property $Z(s, W_L u, W_L v )$.
\end{enumerate}
\end{proposition}
Again, Bruhat order for $\PGB$ may be described using only simple relations 
and we focus on reformulations of those relations.  As for $\BGB$, our 
reformulations apply not just to the simple relations.

\subsection{Weyl Group and Roots}
Here we describe (not necessarily simple) Bruhat relations using the Weyl 
group and roots.
\begin{theorem} \label{PGBBruhat}
Let $w \in W_G$ and let $\alpha \in \Delta^+$.  Then
\begin{eqnarray*}
W_L w &=& W_L w s_\alpha \quad \text{if } w \alpha \in \Delta( \mfl, \mft ) \\
W_L w &\xrightarrow{s_\alpha}& W_L ws_\alpha \quad \text{if } w \alpha \in 
\Delta(\mfn, \mft ) \\
W_L w &\xleftarrow{s_\alpha}& W_L ws_\alpha \quad \text{if } w \alpha \in 
\Delta(\mfn^-, \mft )
\end{eqnarray*}
Note that this is analogous to Bruhat order relations for $\BGB$ with $\mfl$ 
analogous to $\mft$, $\Delta( \mft, \mft ) = \lbrace \rbrace$.
\end{theorem}

\begin{remark}
Compare this characterization to John Stembridge's characterization of 
parabolic Bruhat order in section 2 of \cite{JS} in which $W_L$-cosets are 
associated with $W_G$-orbits in the dual space of the Cartan subalgebra with 
stabilizer $W_L$.  Bruhat order then corresponds to the partial order on the 
root lattice.
\end{remark}

\subsection{Roots and Pullbacks}
Since $\alpha_{\dot{w}} = w \alpha$, we may reformulate Bruhat order as 
follows:
\begin{theorem}
Let $\alpha$ be a positive root and $w \in W_G$.  Then 
\begin{eqnarray*}
P {w} B & \alphato & P{w}{s}_\alpha B
\text{ if } \alpha_{\dot{w}} \in \Delta( \mfn, \mft ) \\
P {w} B & \alphafrom & P {w}{s}_\alpha B \text{ if } 
\alpha_{\dot{w}} \in \Delta( \mfn^-, \mft ).
\end{eqnarray*}
\end{theorem}

\subsection{Cross Actions}
\begin{definition}
The cross action of $W$ on $\PGB$ is the action generated by
$$s_\alpha \times P {w} B := P {w} {s}_\alpha^{-1}B$$
where $\alpha$ is a positive root.
\end{definition}

It follows immediately:
\begin{theorem}
Let $\alpha$ be a positive root and $w \in W$.  Then the cross action
corresponding to $\alpha$ satisfies:
\begin{eqnarray*}
P {w} B &  = & s_\alpha \times P {w} B = 
P{w}{s}_\alpha^{-1} B \text{ if } w\alpha \in \Delta( \mfl, \mft ) \\
P {w} B & \alphato& s_\alpha \times P {w} B = 
P{w}{s}_\alpha^{-1} B \text{ if } w\alpha \in \Delta( \mfn, \mft ) \\
P {w} B & \alphafrom& s_\alpha \times B {w} B = 
P {w}{s}_\alpha^{-1} B \text{ if } w\alpha \in \Delta( \mfn^-, \mft ).
\end{eqnarray*}
\end{theorem}
\begin{proof}
Again, 
if $w \alpha \in \Delta( \mfl, \mft )$, then $\dot{s}_{w\alpha} \in L \subset
P$.
The remainder of the proof resembles previous arguments.
\end{proof}

\section{Reduced Expressions for $\BGB$ and $\PGB$} \label{Reduced}
Throughout this section, 
wherever $w = s_{i_1} s_{i_2}  \cdots s_{i_k}$, 
let $w_j = s_{i_1} s_{i_2}  \cdots s_{i_j}$ for $1 \leq j \leq k$.

\subsection{$\BGB$}
An important aspect of Bruhat order is understanding decompositions of Weyl 
group elements into products of simple reflections.
\begin{definition}
Let $w \in W_G$.  Then the product of simple reflections $w = s_{i_1} s_{i_2} 
\cdots s_{i_k}$ is a 
{\em reduced expression} for $w$ (or {\em $B$-reduced expression}) if $k$ 
is minimal.
\end{definition}

The following result is standard:
\begin{proposition}
Let $w \in W_G$  where $w = s_{i_1} s_{i_2} \cdots s_{i_k}$ as a product of 
simple reflections.  
Then $w =  s_{i_1} s_{i_2} \cdots s_{i_k}$ is a reduced expression if 
and only if $w_j \alpha_{i_{j+1}} > 0$ for $j = 1, 2, \ldots, k-1$. 
\end{proposition}
This is the equivalent definition for reduced expression that generalizes 
nicely to $\PGB$.

\subsection{$\PGB$}
Again, we wish to simplify the existing literature and limit the introduction 
of complex machinery as much as possible.
\begin{definition}
An element $w \in W_G$ is {\em $P$-minimal} if it is a minimal length coset 
representative for $W_L w$.  Equivalently, $w \alpha \in \Delta( \mfn, \mft )$
for every $\alpha \in I$.
\end{definition}

\begin{remark}
As discussed, $W_L \backslash W_G$ is in bijection with the $P$-minimal 
elements of $W_G$.
\end{remark}

\begin{definition}
Let $w \in W_G$ be $P$-minimal.  A {\em $P$-reduced expression} for 
$w$ is a product of simple reflections $w = s_{i_1} s_{i_2} \cdots s_{i_k}$ 
where $w_j \alpha_{i_{j+1}} \in 
\Delta( \mfn, \mft )$ for $j=1, 2, \ldots, k-1$.
We define $\ell^P(w) = k$, the {\em $P$-length} of $w$.
\end{definition}

\begin{lemma}
If $w \in W$ is $P$-minimal, then any $B$-reduced expression for $w$ is also 
$P$-reduced.
\end{lemma}
\begin{proof}
The $P$-minimal element has a $B$-reduced expression $w = s_{i_1} \cdots 
s_{i_k}$.  Each $w_{j-1} \alpha_j$ is positive.  Consider the equations 
\begin{eqnarray*}
W_L w &=& W_L s_\alpha
\qquad \text{if } w \alpha \in \Delta( \mfl, \mft ) \\
W_L w &\xrightarrow{\alpha}& W_L s_\alpha
\qquad \text{if } w \alpha \in \Delta( \mfn, \mft ) \\
W_L w &\xleftarrow{\alpha}& W_L s_\alpha
\qquad \text{if } w \alpha \in \Delta( \mfn^-, \mft ).
\end{eqnarray*}
If the expression is not $P$-reduced, then for some $1 \leq j \leq k$, we have
$w_{j-1}\alpha_j \in \Delta^+( \mfl, \mft )$.  Then $w = s_{w_{j-1}\alpha_j}
w_{j-1}s_{i_{j+1}}
\cdots s_k$ whence $W_L w = W_L s_{i_1} \cdots \hat{s}_{i_{j}} \cdots s_{i_k}$,
contradicting minimality of $w$. Thus our expression must be $P$-reduced.
\end{proof}

\begin{proposition}
Every coset in $W_L \backslash W_G$ has a unique $P$-minimal representative and 
every $P$-minimal representative has a $P$-reduced expression.
\end{proposition}
\begin{proof}
The first statement follows immediately from the definition of $P$-minimal 
element.  The second statement follows from the previous lemma.
\end{proof}

\begin{remark}
We can likewise define $P$-maximal elements:  $w \in W_G$ is $P$-maximal if 
it is a maximal length coset representative for $W_L w$.  Equivalently, 
$w \alpha \in \Delta( \mfn^-, \mft )$ for every $\alpha \in I$.  We also 
have a bijection between $W_L \backslash W_G$ and the $P$-maximal elements of 
$W_G$.  See the discussion of $P$-maximal elements in section \ref{KGP} where 
we discuss $\KGP$.
\end{remark}

\section{Exchange Property for $\BGB$ and $\PGB$} \label{Exchange}
\subsection{$\BGB$}
The Exchange Property for $W_G$ is the following:
\begin{theorem}
Let $w = s_{i_1} s_{i_2} \ldots s_{i_k} \in W_G$ be a reduced expression 
and let $\alpha \in \Pi$.
If $w \alpha < 0$, then  $\ell(w) > \ell(w s_\alpha)$.  
The Exchange Property is the assertion that
there exists some $j$ such that $w s_\alpha = s_{i_1} s_{i_2} \cdots 
\hat{s}_{i_j} \cdots s_{i_k}$ is a reduced expression for $ws_\alpha$ and 
hence $s_{i_1}s_{i_2} \cdots\hat{s}_{i_j} \cdots s_{i_k}s_\alpha$ is a reduced 
expression for $w$.
\end{theorem}

\subsection{$\PGB$}
The Exchange Property for $W_L \backslash W_G$ may be described similarly:
\begin{theorem}
Let $w = s_{i_1} s_{i_2} \cdots s_{i_k} \in W_G$ be a $P$-reduced expression
and let $\alpha \in \Pi$.
If $w \alpha \in \Delta( \mfn^-, \mft)$, then  there exists some $j$ such that 
$w s_\alpha = s_{i_1} s_{i_2} \cdots \hat{s}_{i_j} \cdots s_{i_k}$ and it 
is a $P$-reduced expression for $ws_\alpha$.  It follows that 
$w = s_{i_1} s_{i_2} \cdots \hat{s}_{i_j} \cdots s_{i_k}s_\alpha$ is a 
$P$-reduced expression for $w$.
\end{theorem}
\begin{proof}
We know that there exists $j$ such that 
$s_{i_1} s_{i_2} \cdots \hat{s}_{i_j} \cdots s_{i_k}$ 
is a $B$-reduced expression for $ws_\alpha$.    Since $w \alpha \in 
\Delta( \mfn^-, \mft )$, therefore $PwB \neq Pws_\alpha B$, 
$\dim Pw s_\alpha B = \dim Pws_\alpha B - 1$, and $\ell( ws_\alpha ) = k-1$.  
Since, as we analyze each location in any expression, roots in $\mfl$ fix the 
corresponding orbit,  roots in $\mfn$ increase the orbit dimension, while roots in $\mfn^-$ 
decrease the orbit dimension, therefore each simple reflection in our 
expression must increase dimension, whence
$w s_\alpha = s_{i_1} s_{i_2} \cdots \hat{s}_{i_j} \cdots s_{i_k}$ 
must be a $P$-reduced expression as well.  Since $w = s_{i_1} \cdots s_{i_k}$ 
is a $P$-reduced expression for $w$, considering length, so must $w = s_{i_1} 
\cdots \hat{s}_{i_j} \cdots s_{i_k} s_\alpha$.
\end{proof}

\section{Bruhat Order on $K \backslash G / B$} \label{KGBSection}
Bruhat order for $\KGB$ may differ in ``direction'' in the literature due to a 
preference to associate the minimal length reduced expression with the open 
dense orbit since the open orbit is unique while the closed minimal dimension 
orbits generally are not.
\subsection{Parametrizing $\KGB$}
We use the parametrization of $K \backslash G / B$ as presented in
\cite{S}, which is an excellent reference.
Recall that $K = G^\theta$ and $B = TU$ is a $\theta$-stable Borel subgroup 
and Levi decomposition.
\begin{notation} Modifying Springer's parametrization for 
$B \backslash G / K$, we set:
\begin{itemize}
\item $\sV := \lbrace x \in G | x^{-1}\theta(x) \in N(T) \rbrace$
\item $V := K \times T$ -orbits on $\sV$: $(k,t) \cdot x = kxt^{-1}$
\item $\dot v \in \sV$ is a representative of $v$.
\end{itemize}
\end{notation}

\begin{proposition} \cite{S}
$V$ is in bijection with $\KGB$.
\end{proposition}

\subsection{Real Forms and Root Types}
In order to discuss Bruhat order in detail, we must discuss real forms and 
root types.
A real form of the complex Lie algebra $\g$ is a real Lie subalgebra  $\g_0$
such that $\g = \g_0 \oplus i \g_0$.  A less obvious way to specify a real 
form is to select a Cartan involution $\theta$.  (Use the Cartan decomposition 
$\g = \mfk \oplus \mfs$ and the fact that the Killing form is positive 
definite on $i \mfk_0 \oplus \mfs_0$.)

We begin by studying the Cartan subalgebra.
\begin{lemma}(\cite{K}) \label{roottype}
Given any $\theta$-stable Cartan subalgebra $\mft$ and
Cartan decomposition $\mft_0 = \mft_0^c \oplus \mft_0^n$ of its real form,
it is known that roots $\alpha \in \Delta( \g, \mft)$ are
real-valued on $i \mft_0^c \oplus \mft_0^n$.
Thus $\theta \alpha = - \bar{\alpha}$.
\end{lemma}

\begin{remark}
Recall that $\bar{\alpha}( X ) = \overline{\alpha( \bar{X})}$.
\end{remark}

\begin{definition}
Given $\mft$ a $\theta$-stable CSA, {\em relative to $\theta$}, $\alpha \in
\Delta( \g, \mft )$ is:
\begin{enumerate}
\item real if $\theta \alpha = - \alpha$
\item imaginary if $\theta \alpha = \alpha$
\item complex if $\theta \alpha \neq \pm \alpha$.
\end{enumerate}
\end{definition}

\begin{definition}
Given $g \in G$, recall $\alpha_g := \Ad_{g^{-1}}^* \alpha : \mft_g = \Ad_g
\mft \to \bbC$.  \em{Relative to $g$}, $\alpha$ is:
\begin{enumerate}
\item real if $\bar{\alpha}_g =  -\alpha_g$
\item imaginary if $\bar{\alpha}_g = \alpha_g$
\item complex if $\bar{\alpha}_g \neq \pm \alpha_g$.
\end{enumerate}
\end{definition}

\begin{notation}
Let $v \in V$ and let $n = \dot{v}^{-1}
\theta (\dot{v})$, and $w = nT$.
Since $v \in V$, it follows that $\theta(w) = w^{-1}$.
\end{notation}

We study the particular case where $g=\dot{v} \in \sV$.
\begin{lemma}(\cite{S})
If $v \in V$, then $\dot{v}T\dot{v}^{-1}$ is a $\theta$-stable
Cartan subgroup.
\end{lemma}
This allows us to describe root types using $\theta$.
\begin{definition}
If $v \in V$, then {\em relative to $v$} $\alpha$ is:
\begin{enumerate}
\item real if $\theta{\alpha}_v =  -\alpha_v$
\item imaginary if $\theta{\alpha}_v = \alpha_v$
\item complex if $\theta{\alpha}_v \neq \pm \alpha_v$.
\end{enumerate}
since $T_{\dot{v}}$ is $\theta$-stable.
\end{definition}

Equivalently,
\begin{definition} (\cite{S}) \label{KGBRootTypes}
{\em Relative to $v \in V$ (or $w = \dot{v}^{-1}\theta(\dot{v})T \in W_G$)}
$\alpha$ is:
\begin{enumerate}
\item real if $w\theta \alpha = - \alpha$
\item imaginary if $w\theta \alpha = \alpha$
\item complex if $w\theta \alpha \neq \pm \alpha$.
\end{enumerate}
\end{definition}

\begin{proposition}
The previous two definitions for real, complex, and imaginary are consistent.
\end{proposition}
\begin{proof}
Let $T_1 = vTv^{-1}$, which is $\theta$-stable because $v \in \sV$.  Then
$-\bar{\alpha}_v = \theta \alpha_v$.  Given $t_1 \in T_1$,
\begin{equation*}
\theta\alpha_v(t_1) = \alpha_v(\theta^{-1}(t_1)) = \alpha( v^{-1}
\theta^{-1}(t_1) v )
\end{equation*}
whereas for $t= v^{-1}t_1v \in T$,
\begin{eqnarray*}
w\theta\alpha(t) = \theta\alpha( n^{-1} t n ) &=& \alpha( \theta^{-1}(
\theta(v)^{-1}v )\theta^{-1}(t) \theta^{-1}( v^{-1}\theta(v) ) )  \\
&=&  \alpha( v^{-1} \theta^{-1}(vtv^{-1}) v) \\
&=& \alpha( v^{-1} \theta^{-1}(t_1) v ).
\end{eqnarray*}
Since $\alpha_{v}(t_1) = \alpha(t)$, we see therefore that
\begin{eqnarray*}
\theta\alpha_v(t_1) = \alpha_v(t_1) &\iff& w\theta\alpha(t) =
\alpha(t) \\
\theta\alpha_v(t_1) = -\alpha_v(t_1) &\iff& w\theta\alpha(t) =
-\alpha(t) \\
\theta\alpha_v(t_1) \neq \pm \alpha_v(t_1) &\iff& w\theta
\alpha(t) \neq \pm \alpha(t).
\end{eqnarray*}
\end{proof}

It follows from these computations that:
\begin{corollary} \label{wvCorollary}
For $\alpha \in \Pi$, $v \in V$, $w = \dot{v}^{-1} \theta( \dot{v} ) T \in W$,
$$w \theta \alpha = (\theta \alpha_v )_{v^{-1}}.$$
\end{corollary}

We may further distinguish imaginary roots as compact or noncompact.
\begin{definition}
Let $\alpha$ be an imaginary root.
Normalizing the one-parameter subgroup $x_\alpha$ appropriately, 
$$\theta( x_\alpha(\xi)) = x_{\theta\alpha}(c_\alpha \xi) = x_\alpha( c_\alpha
\xi) \quad\text{where }c_{\alpha} = \pm 1.$$
The root $\alpha$ is said to be {\em compact imaginary} if 
$c_\alpha=1$ and {\em noncompact imaginary} if $c_\alpha = -1$.
\end{definition}

\begin{definition}
Suppose the root $\alpha$ is imaginary relative to $v \in V$. Then,
normalizing $x_{\alpha_v}$ appropriately,
$$\theta( x_{\alpha_v}(\xi)) = x_{\theta{\alpha_v}}(c_{\alpha_v} \xi) = 
x_{\alpha_v}( c_{\alpha_v} \xi) \quad\text{where }c_{\alpha_v} = \pm 1.$$
We say that $\alpha$ is {\em compact relative to $v$} if $c_{\alpha_v} = 1$ 
and {\em noncompact} if $c_{\alpha_v} = -1$.
\end{definition}

Springer showed that if $v \in \sV$ and $n \in N(T)$, then $vn^{-1} \in \sV$, 
so there is a left $W_G$-action on $\sV$ and also on $V$ \cite{S}.
\begin{definition}
The {\em cross action} on $\KGB$ corresponds to Springer's $W_G$-action on $V$.
That is,
$$s_\alpha \times K\dot{v}B = K\dot{v}\dot{s}_\alpha^{-1}B.$$
\end{definition}

Suppose $\alpha \in \Pi$ is noncompact imaginary relative to $v \in V$.
In section 6.7 of \cite{S}, Springer defines the automorphism
$\psi(g) = \dot{v}^{-1} \theta( \dot{v}) \theta(g) \theta(\dot{v})^{-1}\dot{v}$.
We observe that this is simply $\theta_n(g) : = int(n) \circ \theta(g)$.  Since 
$\theta(n) = n^{-1}$, this is an involutive automorphism.
It is now easy to see, as Springer pointed out,
that $\psi$ descends to an involutive automorphism of $G_\alpha$, the subgroup
corresponding to $\pm \alpha$, since $\alpha$ is imaginary relative to
$\theta_n$.
Springer shows that $\psi( x_\alpha(m) ) = x_\alpha(-m)$,
$\psi(x_{-\alpha}(m) ) = x_{-\alpha}(m)$, and $\psi(\dot{s}_\alpha) =
\dot{s}_\alpha^{-1}$.  Springer claims that there is $z_\alpha \in G_\alpha$
such that $z_\alpha \psi(z_\alpha )^{-1} = \dot{s}_\alpha$.  We see we may
choose $z_\alpha =
x_\alpha(-1)x_{-\alpha}(1/2)$ since $z_\alpha\psi(z_\alpha)^{-1} =
x_\alpha(-1)x_{-\alpha}(1/2) x_{-\alpha}(1/2)x_\alpha(-1) = \dot{s}_\alpha$.
Then:
\begin{definition}\label{SpringerCayley}
Given $v \in V$, $\alpha \in \Pi$ noncompact imaginary relative to $v$, the 
{\em Cayley transform} of $v$ through $\alpha$ is
$$c^\alpha(K\dot{v}B) = K \dot{v} z_\alpha^{-1} B.$$
The Cayley transform is known to increase orbit dimension by $1$.
\end{definition}

In section 4.3 of \cite{RS}, one finds that: \\
\begin{tabular}{|l|l|l|}
\hline
Case & Type of $\alpha$ wrt $v$ & $s_\alpha \times v$ \\ \hline
i) & real & $s_\alpha \times v = v$ \\
ii) & compact imaginary & $s_\alpha \times v = v$ \\
iii) & complex & $s_\alpha \times v \neq v$ \\
iv.I) & noncompact imaginary type I & $s_\alpha \times v \neq v$ \\
iv.II) & noncompact imaginary type II & $s_\alpha \times v = v$ \\ \hline
\end{tabular} \\
Also,  \\
\begin{tabular}{|l|l|l|}
\hline
Case & Root Type of $\alpha_v$ & $\pi_\alpha^{-1}\pi_\alpha \cO_v$ \\ \hline
i.I) & real type I & $\cO_v \cup \cO_{v'}$ where $v = c^\alpha( v' )$ \\
i.II) & real type II & $\cO_v \cup \cO_{v'} \cup s_\alpha \times \cO_{v'}$
where $v = c^\alpha( v' ) = c^\alpha( s_\alpha \times v' )$ \\
ii) & compact imaginary & $\cO_v$ \\
iii.a) & complex downward & $\cO_v \cup s_\alpha \times \cO_v$  \\
iii.b) & complex upward & $\cO_v \cup s_\alpha \times \cO_v$  \\
iv.I) & noncompact imaginary type I & $\cO \cup s_\alpha \times \cO \cup
c^\alpha \cO$   \\
iv.II) & noncompact imaginary type II & $\cO \cup c^\alpha \cO$  \\ \hline
\end{tabular} \\
See \cite{PSY} for more detail on the choice of nomenclature. \\

Thus to simplify the definition of type I and type II roots, we define:
\begin{definition}
If the simple root $\alpha$ is noncompact imaginary relative to 
$v \in V$, then $\alpha$ is said to be:
\begin{enumerate}
\item type I relative to $v$ if $s_\alpha \times v \neq v$
\item type II relative to $v$ otherwise.
\end{enumerate}
We also define types I and II for real roots.
If $\alpha$ is real relative to $v$, then $\alpha$ is said to be:
\begin{enumerate}
\item type I relative to $v$ if $\alpha$ is type I relative to $c^\alpha(v)$
\item type II relative to $v$ otherwise.
\end{enumerate}
\end{definition}

We have a richer theory for certain type II roots.  Leading towards such 
results, we consider:
\begin{lemma}(\cite{S}, p. 527) \label{ThetaLemma}
Recall $\dot{s}_\alpha$ was defined using one-parameter subgroups.  Then:
\begin{itemize}
\item[i)] if $\alpha$ is real:
        $\theta( \dot{s}_\alpha ) = \dot{s}_{-\alpha}$
\item[ii)] if $\alpha$ is compact imaginary:
        $\theta(\dot{s}_\alpha) = \dot{s}_{\alpha}$
\item[iii)] if $\alpha$ is complex
        $\theta( \dot{s}_\alpha ) = \dot{s}_{\theta \alpha}$
\item[iv)] if $\alpha$ is noncompact imaginary:
        $\theta(\dot{s}_\alpha) = \dot{s}_{-\alpha}$.
\end{itemize}
\end{lemma}

Recall that $\dot{s}_\alpha = \phi_\alpha \left ( 
\begin{array}{cc} 0 &-1 \\ 1 & 0 \end{array} \right)$.  Since the matrix 
$\left(\begin{array}{cc} 0 &-1 \\ 1 & 0 \end{array} \right)$ has order $4$, 
therefore either $\dot{s}_\alpha^2 = 1$ or $\dot{s}_\alpha^4 = 1$.
\begin{notation}
Let $m_\alpha = \dot{s}_\alpha^2$.
\end{notation}
We are motivated by Lemma 14.11 of \cite{AD} to consider the following:
\begin{proposition} \label{TypeITypeIIProp}
Let $\alpha$ be a simple root such that $m_\alpha = 1$.  Then:
\begin{enumerate}
\item \label{Eqn1} $\dot{s}_{-\alpha} =  \dot{s}_\alpha = \dot{s}_\alpha^{-1}$,
\item $\alpha$ is type II relative to all orbits, and
\item all the roots in the Weyl group orbit of $\alpha$ must be of type II.
\end{enumerate}
\end{proposition}
\begin{proof}
Since
$\dot{s}_\alpha = \phi_\alpha \left( \begin{array}{cc} 0& -1 \\ 1 & 0 
\end{array} \right )$
while
$\dot{s}_{-\alpha} = \phi_\alpha \left( \begin{array}{cc} 0& 1 \\ -1 & 0 
\end{array} \right )$, we see that $\dot{s}_{-\alpha} = \dot{s}_\alpha^{-1}$.
Considering order, we obtain the first equation.

The condition $m_\alpha = 1$ is conjugation-invariant.  That is, $m_{\alpha_g}
= g m_\alpha g^{-1} = 1$ as well for all $g \in G$.  To prove the 
proposition, it suffices to prove that for any $v \in V$ relative to which 
$\alpha$ is noncompact imaginary, $s_\alpha \times K\dot{v}B = K\dot{v}B$. 
Observe that 
$\theta( \dot{s}_{\alpha_v} ) = \dot{s}_{\alpha_v}$ by Lemma \ref{ThetaLemma} 
and (\ref{Eqn1}) whence 
$\dot{s}_{\alpha_v} \in K$.  Then $K\dot{v} \dot{s}_\alpha^{-1} B = K 
\dot{s}_{\alpha_v}^{-1} \dot{v} B = K\dot{v}B$.

\end{proof}

\begin{remark}
The last two statements of the proposition follow directly from Lemma 14.11 
of \cite{AD} as well.
\end{remark}

\subsection{Cross Actions and Cayley Transforms} \label{KGBCC}
Simple relations for Bruhat order on $\KGB$ may be 
described by cross actions and Cayley transforms.

We list the results which may be found in \cite{RS}.

Recall:  either $\dim \pi_\alpha^{-1} \pi_\alpha ( \cO_v ) = \dim \cO_v$ or
$\dim \pi_\alpha^{-1} \pi_\alpha ( \cO_v ) = \dim \cO_v + 1$. \\
\begin{tabular}{|l|l|l|}
\hline 
Case & Root Type of $\alpha_v$ & $\dim \pi_\alpha^{-1}\pi_\alpha \cO_v$ \\ 
\hline
i.I) & real type I & same \\
i.II) & real type II & same \\
ii) & compact imaginary & same \\
iii.a) & complex downward & same \\
iii.b) & complex upward & +1 \\
iv.I) & noncompact imaginary type I & +1  \\
iv.II) & noncompact imaginary type II & +1  \\ \hline
\end{tabular} \\
\begin{tabular}{|l|l|l|l|l|}
\hline
Case & Root Type of $\alpha_v$ & $\pi_\alpha^{-1}\pi_\alpha \cO$ 
& Other Types & Bruhat Relation \\ \hline
iii.b) & complex & $\cO \cup s_\alpha \cO$ & $\alpha$ complex wrt. $s_\alpha 
\cO$ & $\cO \preceq^K_\alpha s_\alpha \cO$ \\
iv.I) & noncpt type I & $\cO \cup s_\alpha \cO \cup c^\alpha
\cO$  & $\alpha$ real type I wrt. $c^\alpha \cO$  &
$\cO \preceq^K_{\alpha} c^\alpha \cO$ \\
& & & 
$\alpha$ noncpt type I wrt. $s_\alpha \cO$ & 
$s_\alpha \cO \preceq^K_{\alpha} c^\alpha \cO$  \\
iv.II) & noncpt type II & $\cO \cup c^\alpha \cO$ 
& $\alpha$ real type II wrt. $c^\alpha \cO$ & 
$\cO \preceq^K_{\alpha} c^\alpha \cO$ \\ \hline
\end{tabular} 

\begin{theorem} \label{KGBCrossCayley}
If $v \preceq^K_{\alpha} v'$, then either:
\begin{itemize}
\item $v' = s_\alpha \times v$ where $\alpha$ is type iii.b) relative to $v$, 
or 
\item $v' = c^\alpha \times v$ where $\alpha$ is type iv) relative to $v$.
\end{itemize}
\begin{tabular}{|l|c|l|l|}
\hline
$\cO_v$ type & $\preceq^K_{\alpha}$ & $\cO_v'$ type & Relationship \\
\hline
iv.I) & $\preceq^K_{\alpha}$ & i.I) & $\cO_{v'} = c^\alpha \cO_v$ \\
iv.II) & $\preceq^K_{\alpha}$ & i.II) & $\cO_{v'} = c^\alpha \cO_v$ \\
iii.b) & $\preceq^K_{\alpha}$ & iii.a) & $\cO_{v'} = s_\alpha
        \cO_v$ \\ \hline
\end{tabular} 
\end{theorem}

Is there a simple means of explaining when we have an increase or a decrease 
in the Bruhat order for $\KGB$?  We will see shortly that the answer is yes.

\subsection{Weyl Group and Roots}
\begin{theorem} \label{KGBBruhat}
Let $v \in V$ and $\alpha \in \Pi$.  Recall that $w = \dot{v}^{-1}
\theta(\dot{v})T \in W_G$.  Simple relations for Bruhat order on 
$\KGB$ may be formulated by the existence of $v' \in V$ such that: \\
\begin{tabular}{|l|l|}
\hline
$\cO_v \preceq^K_{\alpha} \cO_{v'}$ & 
iff $w \theta \alpha > 0$ and $\g_{\alpha_v} \not \subset \mfk$ \\
$\cO_v \succeq^K_{\alpha} \cO_{v'}$ & 
iff $w \theta \alpha < 0$ and $\g_{\alpha_v} \not \subset \mfk$. \\ \hline
\end{tabular} \\
If $\cO_v \preceq^K_\alpha \cO_{v'}$, then $v' = s_\alpha \times v$ if 
$\alpha$ is complex relative to $v$ and $v' = c^\alpha( v)$ if $\alpha$ is 
noncompact relative to $v$.

We note that this description of Bruhat order is analogous to the 
descriptions for $\BGB$ and for $\PGB$ as follows.  The reductive subalgebra 
$\mfk$ plays an analogous role to $\mfl$ in $\PGB$ and to $\mft$ in $\BGB$.
Furthermore, in the cases $\BGB$ and $\PGB$, $\theta = \Id$.
\end{theorem}
\begin{proof}
Consider the following table: \\
\begin{tabular}{|l|l|l|l|}
\hline
Case & Root Type of $\alpha_v$ & $\dim \pi_\alpha^{-1}\pi_\alpha \cO$ &
Combinatorial Description\\ \hline
i.I) & real & same & $w\theta \alpha = -\alpha < 0$ \\
i.II) & real & same & $w\theta \alpha = -\alpha < 0$ \\
ii) & compact imaginary & same & $w\theta \alpha = \alpha > 0$ but
$\g_{\alpha_v} \subset \mfk$ \\
iii.a) & complex & same & $w\theta\alpha < 0$ \\
iii.b) & complex & +1 & $w\theta\alpha > 0$, $\g_{\alpha_v} \not \subset
\mfk$ \\
iv.I) & noncompact imaginary type I & +1 & $w\theta\alpha = \alpha > 0$,
$\g_{\alpha_v} \not \subset \mfk$ \\
iv.II) & noncompact imaginary type II & +1 & $w\theta\alpha = \alpha > 0$,
$\g_{\alpha_v} \not \subset \mfk$
 \\ \hline
\end{tabular} \\
The real and imaginary cases follow immediately from definitions.

In the complex case, either $w \theta \alpha > 0$ or $w \theta\alpha < 0$.
Since $B$ is $\theta$-stable, therefore $\theta \Delta^+ = \Delta^+$.
If $w \theta \alpha > 0$, then  $\theta( w \theta \alpha ) > 0$.  Since
$\theta(w) = w^{-1}$, therefore $\theta( w \theta \alpha ) =
w^{-1} \alpha$.  From $w^{-1} \alpha > 0$, we conclude that
$\ell( s_\alpha w ) = \ell(w) + 1$.  From $\theta \Delta^+ = \Delta^+$,
we also conclude that $\theta \Pi = \Pi$,  whence $\theta\alpha$ is a
simple root.
Since $\alpha$ is complex relative to $v$ so that $w \theta \alpha \neq
\pm \alpha$, therefore $s_\alpha w \theta \alpha > 0$, whence
$\ell( s_\alpha w s_{\theta\alpha} ) = \ell(s_\alpha w) + 1 = \ell(w)+2$.
Similarly, if $w\theta\alpha < 0$, then $\theta( w \theta\alpha ) = w^{-1}
\alpha < 0$ as well and $\ell(  s_\alpha w s_{\theta\alpha} ) =
\ell(s_\alpha w) - 1 = \ell(w)-2$.  The complex case now follows from
the case analysis in 4.3 of \cite{RS}.
\end{proof}

\subsection{Roots and Pullbacks}
\begin{theorem}
Let $v \in V$ and $\alpha \in \Pi$.  Simple relations for Bruhat order on 
$\KGB$ may be formulated by the existence of $v' \in V$ such that: \\
\begin{tabular}{|l|l|}
\hline
$\cO_v \preceq^K_{\alpha} \cO_{v'}$ & 
if $\theta \alpha_v > 0$ (i.e. $\in \Delta^+(\g, \mft)_v$) and 
$\g_{\alpha_v} \not \subset \mfk$ \\
$\cO_v \succeq^K_{\alpha} \cO_{v'}$ & 
if $\theta \alpha_v < 0$ and $\g_{\alpha_v} \not \subset \mfk$. \\ \hline
\end{tabular}
\end{theorem}
\begin{proof}
This follows from the previous theorem and from Corollary \ref{wvCorollary}.
\end{proof}

\subsection{$\KGB$ in More Depth}
We review the discussion of monoids in \cite{RS} and study how we may 
specify elements of $V$ by the monoidal action using our combinatorial results.

\begin{definition} (\cite{RS}, 3.10)
Given the Coxeter group $(W,S)$, the monoid $M(W$) has generators $m(s)$
($s \in S$) and the relations:
\begin{enumerate}
\item $m(s)^2 = m(s)$ $\quad s \in S$; \label{SquareCondition}
\item braid relations:  if $s, t \in S$ are distinct, then
\begin{enumerate}
\item $o(st) = 2k$:  $\quad (m(s) m(t))^k = (m(t)m(s))^k$
\item $o(st) = 2k+1$:  $\quad (m(s)m(t))^k m(s) = (m(t)m(s))^k m(t)$.
\end{enumerate}
\end{enumerate}
\end{definition}

\begin{proposition} (\cite{RS}, 3.10)
\begin{enumerate}
\item If $w = s_1 s_2 \ldots s_\ell$ is a reduced decomposition of $w$, then 
$m(w) := m(s_1) m(s_2) \cdots m(s_\ell) \in M(W)$ is independent of the reduced 
decomposition chosen.
\item $M(W) = \lbrace m(w) : w \in W \rbrace$.
\item $m(w) m(s) = \left\lbrace
\begin{array}{ll}
m(ws) & \text{ if }ws > w \\
m(w) & \text{ if }ws < w.
\end{array}
\right .$
\end{enumerate}
\end{proposition}

\begin{definition} (\cite{RS}, 4.7)
There is an action of the monoid $M(W)$ on $V$:
if $\cO_{v'}$ is the unique dense orbit in $K v P_\alpha$, then 
$m(s_\alpha) v = v'$.  Thus:
\begin{eqnarray*}
\text{If } \cO_v \preceq^K_\alpha \cO_{v'} \text{ then } \quad m(s_\alpha) v 
&=& v'. \\
\text{Otherwise, } \quad m(s_\alpha) v &=& v.
\end{eqnarray*}
\end{definition}

The monoidal action should be thought of in the following way.  When 
considering Weyl group actions, $s \in S$ is self-inverse, so acting twice 
by $s$ should return the original element.  The action of $s$ can both raise 
and lower dimensions.  In contrast, the monoidal action of $s \in S$ on 
$v \in V$ only changes $v$ if a cross action or Cayley transform corresponding 
to $s$ raises the dimension.  Thus repeated monoidal actions of $s$ are the 
same as acting once.  This agrees with $m(s)^2 = m(s)$.  Considering a string 
of simple monoidal actions, we may always remove the simple elements which do 
not raise dimension.  As for the Weyl group action on $\BGB$ and on $\PGB$, 
any element of $V$ can be obtained by $M(W)$ acting on the closed orbits in 
$V$.

\begin{notation}
Let $V_0$ be the set of closed orbits in $V$.
\end{notation}

\begin{definition} (\cite{RS}, 4.1)
The {\em length} of an element of $V$ is defined as follows:
\begin{enumerate}
\item If $v \in V_0$,  then $\ell(v) = 0$.
\item If $v = m(s) u$ where $v \neq u$, then $\ell(v) = \ell(u) + 1$.
\end{enumerate}
\end{definition}

\begin{definition} (\cite{RS}) 
A {\em sequence} in $S$ is $\bfs = (s_1, \ldots, s_k)$.  The {\em length} of 
$\bfs$ is $k$ and $m(\bfs ) = m(s_k) \cdots m(s_1)$.
\end{definition}

\begin{definition} (\cite{RS}, 5.2)
Given $u, v \in V$, write $u \xrightarrow{\alpha} v$ or $u \xrightarrow
{s_\alpha} v$ if there exists $x \in V$
and a sequence $\bfs \in S$ such that
\begin{enumerate}
\item $u = m(\bfs) x$ and $\ell(u) = \ell(x) + \ell(\bfs )$;
\item $v = m(\bfs)m(s_\alpha) x$ and $\ell(v) = \ell(x) + \ell(\bfs ) + 1$.
\end{enumerate}
The relation defined by
$$u \leq v \text{ if there exists a sequence } u = v_0 \xrightarrow{\alpha_1}
 v_1 \cdots \xrightarrow{\alpha_k} v_k = v$$
is the {\em standard order} on $V$.
\end{definition}

Richardson and Springer show in \cite{RS} that Bruhat order on $\KGB$ 
and standard order are the same.

The inverse of a cross action is single valued.  The inverse of a type II 
Cayley transform is single valued while the inverse of a type I Cayley 
transform is double valued.  We wish to understand how elements of $\KGB$ 
may be identified using sequences in $S$.

\begin{proposition}
Given a sequence in $V$
$$v_0 \xrightarrow{s_1} v_1 \xrightarrow{s_2} v_2 \cdots
 \xrightarrow{s_k} v_k$$
 where $\alpha_k$ is noncompact type I relative to $v_{k-1}$, there is a 
 sequence
$$u_0 \xrightarrow{s_1} u_1 \xrightarrow{s_2} u_2 \cdots \xrightarrow {s_{k-1}}
u_{k-1} = s_{\alpha_k} \times v_{k-1} \xrightarrow{s_k} v_k = 
c^{\alpha_k}(v_{k-1})$$
with each $\alpha_j$ the same types relative to $v_{j-1}$ and to $u_{j-1}$
(eg. $\alpha_k$ is noncompact type I relative to both $u_{k-1}$ and 
$v_{k-1}$).
\end{proposition}
\begin{proof}
Begin by letting $w_j = v_j^{-1} \theta( v_j)T \in W_G$.  Since $\alpha_k$ is 
noncompact imaginary relative to $v_{k-1}$, therefore $w_{k-1} \theta \alpha_k
= \alpha_k$.  Therefore $(v_{k-1}{s}_{\alpha_k}^{-1})^{-1} \theta( v_{k-1}
{s}_{\alpha_k}^{-1}) T = s_{\alpha_k} w_{k-1} \theta( s_{\alpha_k}) = 
s_{\alpha_k} s_{w_{k-1}  \theta \alpha_k} w_{k-1} = s_{\alpha_k}^2 w_{k-1}
 = w_{k-1}$.  Recall Definition  \ref{KGBRootTypes}.
Thus if $\beta$ is real, imaginary, or complex relative to 
$v$ and $\alpha$ is non-compact relative to $v$, then $\beta$ is real, 
imaginary, or complex, respectively, relative to $s_\alpha \times v$.  
Our tables before Theorem \ref{KGBCrossCayley} show that if $\beta$
is type I relative to $v$, then it is type I relative to $s_\alpha 
\times v$.  If $\beta$ is compact relative to $v$ (that is,
$d( \theta int(v)) X_\beta = d \, int(v) X_\beta$ for $X_\beta \in \g_\beta$) 
and $\alpha$ is non-compact type I or type II relative to $v$, 
then we see that $\beta$ is compact relative to $s_\alpha \times v$:
\begin{eqnarray*}
d( \theta int( v \dot{s}_\alpha^{-1} ) ) X_\beta &=&
d( \theta( int( v \dot{s}_\alpha^{-1} v^{-1}) int(v)) X_\beta \\
&=& d( int( \theta( \dot{s}_{\alpha_v} ) \theta int(v)) X_\beta \\
&=& d( int( \dot{s}_{\alpha_v}) int(v) ) X_\beta \\
&=& d( int( v\dot{s}_\alpha ) ) X_\beta \\
&=& d( int( v \dot{s}_\alpha^{-1} )) X_\beta.
\end{eqnarray*}
We see that whatever type some simple 
root $\beta$ is relative to $v_{k-1}$, it is precisely the same type relative 
to $s_{\alpha_k} \times v_{k-1}$.  Recall that $c^{\alpha_k} ( s_{\alpha_k} 
\times v_{k-1} ) = v_k$ as well.  Since only noncompact type I roots cause 
ambiguity in taking inverses of cross actions and Cayley transforms, therefore 
by induction, the proposition holds.
\end{proof}

Thus using our simple combinatorial descriptions of simple relations in 
Bruhat order, it is easy to understand:
\begin{remark}
There are two general 
methods of specifying any element $u \in V$ up to braid relations:
\begin{enumerate}
\item   There is $u_0 \in V_0$ and a sequence
$$u_0 \xrightarrow{s_1} u_1 \xrightarrow{s_2} u_2 \cdots
 \xrightarrow{u_k} u_k = u.$$
This specifies $u$ unambiguously.
\item Let the unique open dense orbit in $\KGB$ be $K v B$.  There is a 
sequence
$$ v = u_\ell \xleftarrow{s_\ell} u_{\ell-1} \xleftarrow{s_{\ell-1}} u_{\ell-2}
\cdots \xleftarrow{s_{k+1}} u_k = u.$$
A sequence moving downwards from the open orbit does not necessarily uniquely 
identify 
the orbit $u$ since the inverse Cayley transform is double valued for type I 
roots.  To uniquely identify $u$, specify a choice for each type I inverse 
Cayley transform.
\end{enumerate}
\end{remark}

\begin{corollary} \label{KGBWUnique}
If $u \in V$ and $w_1, w_2 \in W$ are minimal length satisfying $m(w_1)u = 
m(w_2)u$, then $w_1 = w_2$.
\end{corollary}

\section{Bruhat Order on $\KGP$} \label{KGP}
\subsection{Closure Order and the Order Induced From Bruhat Order on $\KGB$}
Recall that Bruhat order on $\KGP$ is defined to be closure order.
\begin{proposition}\label{KGPClosure}
Closure order on $\KGP$ is the same as the partial order induced from Bruhat 
order on $\KGB$.  That is, writing $K u P \leq K v P$ if there are orbit 
representatives $u_0$ and $v_0$, respectively, such that $K u_0 B \leq K v_0 B$,
$$KuP \leq KvP \iff KuP \preceq^K KvP.$$
\end{proposition}
\begin{proof}
$\Leftarrow$: Since 
$KuB \subset \pi_I^{-1} \pi_I( KuB ) \subset \pi_I^{-1}\pi_I( \overline{KvB} )
\subset \overline{K v_0 B }$
where $K v_0 B$ is the unique dense orbit in $\pi_I^{-1}(KvP)$ (which exists, 
as we 
will see in Corollary \ref{KGPUniqueMaximal}), we see that $K u B \subset 
\overline{Kv_0 B}$. \\
$\Rightarrow$:  We observe that
$
K u_0 B \subset \overline{K v_0 B} 
\Rightarrow KuP = \pi_I( Ku_0 B ) \subset \overline {\pi_I( K v_0 B ) } = 
\overline{ KvP}.$
\end{proof}

\subsection{Understanding $KvP$:  $I$-Equivalence}
We wish to find a simple parametrization of $\KGP$.  As we will see, the key 
to parametrizing 
$\KGP$ is understanding the Bruhat order of $\KGB$ restricted to the $B$-orbits
in a $P$-orbit.

Since $P \supset B$, therefore each $P$-orbit $KvP$ can be expressed as a 
union of $B$-orbits $KuB$.  This is an example of an $I$-equivalence class, 
defined in the preprint \cite{PSY} on generalized Harish-Chandra 
modules:
\begin{definition}
Recall the map $\pi_I:  G / B \to G / P$, the natural projection from the flag 
variety to the partial flag variety of parabolic subgroups of type $I$.
Two orbits $\cO, \cO'$ in $\KGB$ are {\em $I$-equivalent} (write $\cO \sim_I 
\cO'$) if they project to the same $K$-orbit on $G / P$; i.e. $\pi_I( \cO )
= \pi_I( \cO' )$.  The $I$-equivalence class of $\cO$ is
$$[\cO]_{\sim_I} = K \backslash \pi_I^{-1} ( \pi_I(\cO)).$$
\end{definition}

In \cite{PSY}, each $I$-equivalence class $K v_0 P = \cup_{\cO \sim_I 
\cO_{v_0}} \cO$ is shown to be in bijection with some double coset space 
$\Mv \backslash L / B \cap L$.  The idea of considering such a bijection is 
due to Lusztig-Vogan, according to \cite{Ma}, p. 313.  
Note that these are $\Mv$-orbits on $L / B \cap L$, the flag variety of $L$.  
The bijection permits a 
generalization of the following commonly used technique:  when computing 
Kazhdan-Lusztig polynomials, often, the first step is to first make use 
of polynomials arising from smaller root subsystems.  For example, to compute 
type $A_3$ polynomials, begin by finding copies of $A_2$ within $A_3$.
Furthermore, the subgroup $\Mv$ is a spherical subgroup of 
$L$ and thus there is a unique open dense orbit in 
$\Mv \backslash L / B \cap L$.  The bijection of $I$-equivalence classes with 
double coset spaces respects Bruhat order. Therefore, each $I$-equivalence 
class has a unique maximal element since $\Mv \backslash L / B \cap L$ has a 
unique maximal element.  These maximal elements are easy to specify 
combinatorially, giving us a succinct parametrization of $\KGP$.  We now 
proceed to provide 
more details.  Because we study orbits of different subgroups on different 
flag varieties, we use superscripts to differentiate the different orbit types
by subgroup.

\begin{definition} (\cite{PSY})
Consider a parabolic subgroup and Levi decomposition $\mathtt{P=LN}$ where 
$\mathtt{L}$ carries an involution $\mathtt{\Theta}$ (which may not be defined 
on $G$).  A mixed subgroup of 
$G$ is a subgroup of the form $\mathtt{M = L^\Theta N}$.
\end{definition}
\begin{remark}
\item Mixed subgroups generalize $K$, $B$, and $P$ as follows.  Select 
$\mathtt{P}=G$ and $\mathtt{\Theta} = \theta$, then $\mathtt{M}=K$.
Select $\mathtt{P}=B$ and $\mathtt{\Theta} = \mathrm{Id}$, then $\mathtt{M}=B$.
Select $\mathtt{P}=P$ and $\mathtt{\Theta} = \mathrm{Id}$, then $\mathtt{M}=P$.
\end{remark}

\begin{proposition}\label{MGBParam} (\cite{PSY})
There is a bijection $\mathtt{M} \backslash G / B \leftrightarrow \mathtt{
L^\Theta \backslash L / B \cap L}
\times W_{\mathtt{L}} \backslash W_G$.  More specifically, 
$$\mathtt{M} \backslash G / B \leftrightarrow \mathtt{L^\Theta \backslash L /
B \cap L}  \times_{W_{\mathtt{L}}} W_G$$
where the fibre product is with respect to the cross action of $W_{\mathtt{L}}$
on $\mathtt{L^\Theta \backslash L / B \cap L}$.
\end{proposition}
\begin{remark}
Our double cosets are in bijection with a smaller $\KGB$ cross a Weyl group 
quotient.
\end{remark}

\begin{corollary}
Mixed subgroups are spherical subgroups.
\end{corollary}

\begin{definition} (\cite{PSY})
Cross actions and Cayley transforms on $\mathtt{M} \backslash G / B$ may 
be defined by multiplying orbit representatives on the right by 
$\dot{s}_\alpha^{-1}$ and by $z_\alpha^{-1}$, respectively, as before.
\end{definition}

As mentioned, an application of the theory of mixed subgroups is a bijection 
between orbits in an $I$-equivalence class and mixed subgroup orbits on the 
flag variety for a Levi subgroup. 
\begin{theorem}\label{IEquiv} (\cite{PSY})
Given an $I$-equivalence class $[ \cO_{g}^{\mathtt{M}}]_{\sim I}$ of  
$\mathtt{M}$-orbits 
on $G/B$, there exists a mixed subgroup $\Mg$ of $L$ and a bijection  
$$\Mg \backslash L / B \cap L \xrightarrow{\psi} 
\mathtt{M} \backslash g P / B = [ \cO_{g}^{\mathtt{M}}]_{\sim_I}$$
such that the following diagram commutes:
\begin{diagram}
L & \rTo & \Mg \backslash L / B \cap L \\
& \rdTo & \dTo_\psi \\
& & \mathtt{M} \backslash g P / B.
\end{diagram}
The unlabelled maps are the natural maps arising by choosing orbit 
representatives from $L$ and from $g L$.  Furthermore, 
\begin{eqnarray*}
\psi \left( s_\alpha \times \cO^{\Mg}_\ell \right) &=& s_\alpha \times 
\psi\left( \cO^{\Mg}_\ell \right) \\
\text{and} \qquad \psi\left( c^\alpha \left( \cO^{\Mg}_\ell \right)\right) 
&=& c^\alpha \left(  \psi\left( \cO^{\Mg}_\ell \right)\right).
\end{eqnarray*}

In the case where $\mathtt{M} = K$, we may set:
\begin{enumerate}
\item $g = v_0 \in \sV$ to be a representative for $[\cO^K_{v_0}]_{\sim_I}$ 
of minimal dimension
\item $\tilde{\theta} = \mathrm{int}( v_0^{-1}) \circ \theta \circ 
\mathrm{int}( v_0)$
\item $J = \lbrace \alpha \in S : \alpha_{v_0} \text{ is real or imaginary}
\rbrace \cup \lbrace \alpha \in S: \alpha_{v_0} \text{ is complex and }
\theta(\alpha_{v_0}) \in S_{v_0} \rbrace.$
\item $P_J^I = L_J N^I_J$ the $T$-stable Levi decomposition of the parabolic 
subgroup of $L$ corresponding to $J$
\item $\Mv = L_J^{\tilde{\theta}} N^I_J$.
\end{enumerate}
\end{theorem}

\begin{remark}
$I$-equivalence permits us to decompose any $M \backslash G / B$ into unions of
smaller mixed subgroup double coset spaces.  In particular, iterating 
$I$-equivalence to simplify computations does not introduce any type of 
subgroup beyond mixed subgroups.  For this reason and since we may develop a 
rich theory for mixed subgroups (parametrizing orbits and understanding 
Bruhat order very explicitly), we choose to use $\Mg$ in bijections even 
though there are other subgroups for which bijections with the orbits in an 
$I$-equivalence class are simpler to prove.
\end{remark}

\begin{remark}
Compare this theorem and Proposition \ref{MGBParam}
with Brion and Helminck's parametrization of an $I$-equivalence 
class in the symmetric case, i.e. the $B$-orbits in $KgP$, in Proposition 4 
of section 1.5 of \cite{BH}.  They  set $\sV_{g}$ to 
be $\lbrace x \in L \cap \tilde{\theta}(L) : x^{-1} \tilde{\theta}( x ) \in 
N_L(T) \cap \tilde{\theta}(L) \rbrace$ and $N_{g} = \lbrace n \in 
N_L(T) : B \cap L \cap \tilde{\theta}(L) \subset n( B \cap L)n^{-1} \rbrace$.  
Then $$L^{\tilde{\theta}} \backslash \sV_{g} / T \times N_{g} / T 
\leftrightarrow [ \cO^K_{g}]_{\sim I}.$$
The first term in the product is a smaller $\KGB$ while Brion-Helminck show 
the second term to be in bijection with orbits on $L / B \cap L$ of the 
semidirect product of the unipotent radical of $\tilde{\theta}(P) \cap L$ 
with $L^{\tilde{\theta}}$.  That subgroup is not usually a mixed subgroup nor 
a parabolic subgroup of $L$.  Brion and Helminck do not impose the condition 
that $\cO_g^K$ is a minimal dimension equivalence class representative.
\end{remark}

Again, we saw that mixed subgroups are spherical subgroups.
Thus:
\begin{corollary} \label{KGPUniqueMaximal}
Each $I$-equivalence class of orbits has a unique orbit maximal with respect 
to Bruhat order.  Thus each $P$-orbit $KvP$ contains a unique dense $B$-orbit.
\end{corollary}
\begin{remark}
This is equivalent to Proposition 2 of section 1.2 of \cite{BH}.
\end{remark}

\subsection{Parametrizing $\KGP$}
There is a simple combinatorial parametrization of the unique dense orbits 
in each $I$-equivalence class, and hence of $\KGP$.
\begin{theorem} \label{KGPParam}
Let $I \subset \Pi$ correspond to the standard parabolic $P$.
Then the double coset space $\KGP$ is in bijection with $V_P$ where
\begin{eqnarray*}
V_P &:=& \lbrace v \in V : \text{ for every } \alpha \in I, 
w \theta \alpha < 0 \text{ where } w = v^{-1} \theta(v)T \rbrace \\
&=& \lbrace v \in V : \text{ for every } \alpha \in I, m(s_\alpha) v = v 
\rbrace.
\end{eqnarray*}
In other words, $\KGP$ is in bijection with the $I$-maximal elements of 
$V$.
\end{theorem}
\begin{proof}
This follows immediately from the proposition and the corollary above and our 
characterization of Bruhat order for $\KGB$.
\end{proof}

\begin{remark}
In comparison, $\PGB$ is in bijection with $W^I := \lbrace w \in 
W_G : w \alpha > 0 \text{ for every } \alpha \in I \rbrace$, the $P$-minimal 
elements of $W_G$.  As discussed,
$\PGB$ is in bijection with the unique maximal length coset representatives 
as well, giving us a parametrization analogous to our parametrization of 
$\KGP$.  We may think of $\KGP$ as the $P$-maximal elements of $\KGB$.
\end{remark}

\begin{proposition}\label{KGPKGBEquiv} (cf. Proposition \ref{PGBBGBEquiv})
Let $u, v \in V_P$.  Then
$$ KuP \preceq KvP \iff KuB \preceq KvB.$$
\end{proposition}
\begin{proof}
$\Leftarrow$:  This is clear from Lemma \ref{TopologyLemma}. \\
$\Rightarrow$:  Since $KvB$ is dense in $\pi_I^{-1}( KvP )$,
$ KuB \subset \pi_I^{-1}( K uP ) \subset \pi_I^{-1}\overline{KvP} \subset 
\overline{KvB}.$
\end{proof}

\begin{remark}
This short topological proof works for $\PGB$ as well.
\end{remark}

\subsection{Behaviour of Simple Relations: Descent of the Monoidal Action}
Since Bruhat order for $\KGP$ is induced from Bruhat order on $\KGB$, which can 
be described using simple relations, one concludes that Bruhat order for $\KGP$ 
can be described using simple relations as well.  However, the absence of a 
Borel subgroup among the two subgroups with respect to which we take double 
cosets complicates matters somewhat, obstructing the possibility of making a 
natural definition for $\xrightarrow{\alpha}$ consistent among all coset 
representatives.

\begin{proposition}
\begin{enumerate}
\item If $\alpha \in I$, then $\dot{s}_\alpha, z_\alpha \in L \subset P$; thus
$$\pi_I( v ) = \pi_I( m( s_\alpha ) v )\quad \text{for all } v \in V.$$
\item If $\alpha \in \Pi \setminus I$, 
$$\pi_I( v ) \neq \pi_I( m( s_\alpha ) v) \iff v \neq m(s_\alpha) v.$$
\end{enumerate}
\end{proposition}
\begin{proof}
This follows immediately from Proposition \ref{IEquiv}.
\end{proof}

Thus we may restrict our attention to simple relations in $\KGB$ for $\alpha 
\in \Pi \setminus I$.  

We consider defining cross action to be $s_\alpha \times KvP = K v
\dot{s}_\alpha^{-1} P$.  Since for $w \in 
W_L$,
\begin{eqnarray*}
K v B & \xrightarrow{s_\alpha \times} & Kv \dot{s}_\alpha^{-1} B  \subset K v
\dot{s}_\alpha^{-1} P \\
K vw B & \xrightarrow{s_\alpha \times} & K v w \dot{s}_\alpha^{-1}B \\
K vw B & \xrightarrow{s_{w^{-1}\alpha}\times} & K v \dot{s}_\alpha^{-1} w B 
\subset Kv \dot{s}_\alpha^{-1} P,
\end{eqnarray*}
we see that the cross action does not descend naturally from $\KGB$ to $\KGP$.

\begin{lemma} \label{AlphaLemma}
For $\alpha \in \Pi \setminus I$, $L$ normalizes $N$, so $W_L \alpha \subset 
\Delta( \mfn, \mft )$.
\begin{enumerate}
\item  For any $w \in W_L$, the coefficient of $\alpha$ in the expression of 
$w\alpha$ as a linear combination of simple roots is $1$.
\item  If $\beta \in W_L \alpha$ is a simple root, then $\beta = \alpha$.
\end{enumerate}
\end{lemma}

\begin{proposition} \label{KGPMonoid}
If $w \in W_L$, $\alpha \in \Pi \setminus I$, and $v \in V_P$, then
$$\cO^P_{m(s_{w\alpha})v} = \cO^P_{m(s_\alpha)v}.$$ 
\end{proposition}
\begin{proof}
By Lemma \ref{AlphaLemma} and Lemma 5.3.3 of \cite{Y}, there is a reduced 
expression of 
$s_{w \alpha}$ of the form $s_1 \cdots s_k s_\alpha s_k \cdots s_1$ where the 
$s_i \in W_L$.  Since $v$ is maximal, $m(s_k \cdots s_1 ) v = v$.  Then 
$m( s_{w\alpha} ) v = m( s_1 \cdots s_k ) m( s_\alpha ) m(s_k \cdots s_1) v
= m( s_1 \cdots s_k ) m(s_\alpha)v$.  Since the monoidal action by $W_L$ preserves 
$P$-orbits, the proposition follows.
\end{proof}

\begin{proposition}
Let $v \in V_P$ and $u \in V$ with $u \sim_I v$.
Let $w \in W_L$ be of minimal length such that $v = m(w) u$.
If $\alpha \in \Pi \setminus I$, then $\cO^P_{m(s_\alpha)v} = \cO^P_{
m(s_{w^{-1}\alpha})v} =\cO^P_{ m(s_{w^{-1}\alpha}) u}$.
\end{proposition}

\begin{remark}
It is tempting at this point, but incorrect, to conclude that the monoidal action of $W_G$ 
on $\KGB$ descends naturally to a monoidal action on $\KGP$ as follows:
\begin{itemize}
\item $m(s_{w\alpha}) v = m( w s_\alpha w^{-1} ) m(w) u \stackrel{?}{=} 
m( w ) m(s_\alpha) u$.
\item $\cO^P_{m(s_{w\alpha})v} = \cO^P_{m(s_\alpha)v}$ and 
$\cO^P_{m(w)m(s_\alpha) u} \cO^P_{m(s_\alpha) u}$ so by Proposition 
\ref{KGPMonoid}, $\cO^P_{m(s_\alpha)v} = \cO^P_{m(s_\alpha) u}$.
\end{itemize}
However, we cannot cancel inverses in $M(W_G)$, so the above argument is 
incorrect.  It is easy to find a rank two counterexample for which 
$\cO^P_{m(s_\alpha)v} \neq \cO^P_{m(s_\alpha)u}$.
\end{remark}

\begin{proposition}
If $v \in V_P$, $\alpha$ and $\beta \in \Pi \setminus I$, and $\alpha \neq 
\beta$ with $v \xrightarrow{\alpha} m(s_\alpha)v$ and $v \xrightarrow{\beta}
m(s_\beta) v$, then 
$$\cO^P_{m(s_\alpha) v} \neq \cO^P_{m(s_\beta) v}.$$
\end{proposition}
\begin{proof}
Assume by contradiction that $m(s_\alpha)v$ and $m(s_\beta)v$ belong to the 
same $P$-orbit.  Then there exist minimal length elements $w_\alpha, w_\beta 
\in W_L$ such that 
$m(w_\alpha s_\alpha)v = m(w_\alpha) m(s_\alpha) v = m(w_\beta ) m(s_\beta) v 
= m( w_\beta s_\beta ) v \in V^P$.
Then by Corollary \ref{KGBWUnique}, $w_\alpha s_\alpha = w_\beta s_\beta$
which implies that $s_\alpha s_\beta = 
w_\alpha^{-1} w_\beta \in W_L$, giving $s_\alpha = 
s_\beta$--contradiction.
\end{proof}

\section{Conclusion}
It would be interesting to apply the simplifications of Bruhat order to the 
study of Kazhdan-Lusztig-Vogan polynomials.  The theory of parabolic 
Kazhdan-Lusztig polynomials appears the most likely to benefit from the 
simplifications.

Another topic for future consideration is to further explore the philosophy 
of proving results for $\PGB$ and for $\KGB$ by reducing to $\BGB$ using our 
analogies for simple relations.  For example, can it be applied to 
develop a better understanding of the exchange 
property and the deletion condition for $\KGB$?

Can the theory for $\KGB$ be simplified by using the Tits group?

Can the theories for $\KGP$ and $\PGB$ be made more similar by recasting 
results for $\PGB$ using maximal length representatives rather than 
minimal length representatives?

The reader will find more material on Bruhat order in \cite{PSY}.  In 
particular, it contains a description of Bruhat order for mixed subgroups 
(for which parabolic subgroups and symmetric subgroups are a special case) 
and for situations where one of the subgroups with respect to which we 
take double cosets is twisted by conjugation.
The descriptions of Bruhat order through pullbacks 
of roots in particular carries over to the twisted case very naturally.

\nocite{CC}
\nocite{H}
\nocite{K}
\nocite{S2}
\nocite{H3}
\nocite{H4}
\nocite{A}
\nocite{D}
\nocite{LV}
\nocite{V6}
\nocite{BBr}
\nocite{RW}

\bibliographystyle{alpha}
\bibliography{bruhat}

\begin{thebibliography}{Hum95}

\bibitem[AD09]{AD}
Jeffrey~D. Adams and Fokko DuCloux.
\newblock Algorithms for the representation theory of real reductive {L}ie
  groups.
\newblock {\em Journal of the Institute of Mathematics of Jussieu}, 8:209--259,
  2009.

\bibitem[Ada]{A}
Jeffrey~D. Adams.
\newblock The {A}tlas algorithm: Workshop, {S}alt {L}ake {C}ity.
\newblock \url{http://www.liegroups.org/workshopNotes/adams.pdf}.

\bibitem[BB05]{BBr}
Anders Bj{\"o}rner and Francesco Brenti.
\newblock {\em Combinatorics of {C}oxeter groups}, volume 231 of {\em Graduate
  Texts in Mathematics}.
\newblock Springer, New York, 2005.

\bibitem[BH00]{BH}
Michel Brion and Aloysius~G. Helminck.
\newblock On orbit closures of symmetric subgroups in flag varieties.
\newblock {\em Canad. J. Math.}, 52(2):265--292, 2000.

\bibitem[CC87]{CC}
Luis~G. Casian and David~H. Collingwood.
\newblock The {K}azhdan-{L}usztig conjecture for generalized {V}erma modules.
\newblock {\em Math. Z.}, 195(4):581--600, 1987.

\bibitem[dC]{D}
Fokko du~Cloux.
\newblock Combinatorics for the representation theory of real reductive groups.
\newblock \url{http://www.liegroups.org/papers/summer05/combinatorics.ps}.

\bibitem[Deo77]{D2}
Vinay~V. Deodhar.
\newblock Some characterizations of {B}ruhat ordering on a {C}oxeter group and
  determination of the relative {M}\"obius function.
\newblock {\em Invent. Math.}, 39(2):187--198, 1977.

\bibitem[DM91]{DM}
Fran{\c{c}}ois Digne and Jean Michel.
\newblock {\em Representations of finite groups of {L}ie type}, volume~21 of
  {\em London Mathematical Society Student Texts}.
\newblock Cambridge University Press, Cambridge, 1991.

\bibitem[Hum75]{H3}
James~E. Humphreys.
\newblock {\em Linear algebraic groups}.
\newblock Springer-Verlag, New York, 1975.
\newblock Graduate Texts in Mathematics, No. 21.

\bibitem[Hum90]{H}
James~E.\ Humphreys.
\newblock {\em Reflection groups and {C}oxeter groups}.
\newblock Number~29 in Cambridge Studies in Advanced Mathematics. Cambridge
  University Press, Cambridge, 1990.

\bibitem[Hum95]{H4}
James~E. Humphreys.
\newblock {\em Conjugacy classes in semisimple algebraic groups}, volume~43 of
  {\em Mathematical Surveys and Monographs}.
\newblock American Mathematical Society, Providence, RI, 1995.

\bibitem[Kna02]{K}
Anthony~W. Knapp.
\newblock {\em {L}ie groups beyond an introduction}, volume 140 of {\em
  Progress in Mathematics}.
\newblock Birkh\"auser Boston Inc., Boston, MA, second edition, 2002.

\bibitem[Kno95]{Kno}
Friedrich Knop.
\newblock On the set of orbits for a {B}orel subgroup.
\newblock {\em Comment. Math. Helv.}, 70(2):285--309, 1995.

\bibitem[LV83]{LV}
George Lusztig and David~A. Vogan, Jr.
\newblock Singularities of closures of {$K$}-orbits on flag manifolds.
\newblock {\em Invent. Math.}, 71(2):365--379, 1983.

\bibitem[Mat82]{Ma}
Toshihiko Matsuki.
\newblock Orbits on affine symmetric spaces under the action of parabolic
  subgroups.
\newblock {\em Hiroshima Math. J.}, 12(2):307--320, 1982.

\bibitem[PSY]{PSY}
Annegret Paul, Siddhartha Sahi, and Wai~Ling Yee.
\newblock Generalized {H}arish-{C}handra modules.
\newblock {\em Preprint}.

\bibitem[RS90]{RS}
R.~W. Richardson and T.~A. Springer.
\newblock The {B}ruhat order on symmetric varieties.
\newblock {\em Geom. Dedicata}, 35(1-3):389--436, 1990.

\bibitem[RW10]{RW}
Konstanze Rietsch and Lauren Williams.
\newblock Discrete {M}orse theory for totally non-negative flag varieties.
\newblock {\em Adv. Math.}, 223(6):1855--1884, 2010.

\bibitem[Spr85]{S}
T.~A. Springer.
\newblock Some results on algebraic groups with involutions.
\newblock In {\em Algebraic groups and related topics ({K}yoto/{N}agoya,
  1983)}, volume~6 of {\em Adv. Stud. Pure Math.}, pages 525--543.
  North-Holland, Amsterdam, 1985.

\bibitem[Spr09]{S2}
T.~A. Springer.
\newblock {\em Linear algebraic groups}.
\newblock Modern Birkh\"auser Classics. Birkh\"auser Boston Inc., Boston, MA,
  second edition, 2009.

\bibitem[Ste74]{St}
Robert Steinberg.
\newblock {\em Conjugacy classes in algebraic groups}.
\newblock Lecture Notes in Mathematics, Vol. 366. Springer-Verlag, Berlin,
  1974.
\newblock Notes by Vinay V. Deodhar.

\bibitem[Ste02]{JS}
John~R. Stembridge.
\newblock A weighted enumeration of maximal chains in the {B}ruhat order.
\newblock {\em J. Algebraic Combin.}, 15(3):291--301, 2002.

\bibitem[Vog83]{V6}
David~A. Vogan.
\newblock Irreducible characters of semisimple {L}ie groups {III}. {P}roof of
  {K}azhdan-{L}usztig conjecture in the integral case.
\newblock {\em Invent. Math.}, 71(2):381--417, 1983.

\bibitem[Yee05]{Y}
Wai~Ling Yee.
\newblock The signature of the {S}hapovalov form on irreducible {V}erma
  modules.
\newblock {\em Represent. Theory}, 9:638--677 (electronic), 2005.

\end{thebibliography}
\end{document}